\documentclass[12pt,oneside,english,jou]{amsart}
\usepackage[T1]{fontenc}
\usepackage[latin9]{inputenc}
\usepackage{babel}
\usepackage{verbatim}
\usepackage{mathrsfs}
\usepackage{amstext}
\usepackage{amsthm}
\usepackage{amssymb}
\usepackage{enumerate}
\usepackage[unicode=true,pdfusetitle,
bookmarks=true,bookmarksnumbered=false,bookmarksopen=false,
breaklinks=false,pdfborder={0 0 1},colorlinks=false]
{hyperref}
\usepackage[margin=2.5cm]{geometry}
\usepackage[foot]{amsaddr}

\makeatletter
\numberwithin{equation}{section}
\numberwithin{figure}{section}
\theoremstyle{plain}
\newtheorem{thm}{\protect\theoremname}[section]
\theoremstyle{definition}
\newtheorem{rem}[thm]{\protect\remarkname}
\theoremstyle{definition}
\newtheorem{defn}[thm]{\protect\definitionname}
\theoremstyle{plain}
\newtheorem{prop}[thm]{\protect\propositionname}
\theoremstyle{plain}
\newtheorem{lem}[thm]{\protect\lemmaname}
\theoremstyle{plain}

\theoremstyle{plain}
\newtheorem{cor}[thm]{\protect\corollaryname}

\theoremstyle{definition}

\theoremstyle{definition}

\theoremstyle{definition}
\newtheorem{examplex}[thm]{\protect\examplename}

\usepackage{tikz}
\usetikzlibrary{shapes,arrows}
\usepackage{verbatim}
\usepackage{amsthm}
\usepackage{amstext}

\newcommand{\eps}{\varepsilon}
\newcommand{\R}{\mathbb{R}}
\newcommand{\N}{\mathbb{N}}
\newcommand{\Z}{\mathbb{Z}}

\newcommand{\mF}{\mathcal{F}}
\newcommand{\mH}{\mathcal{H}}

\newcommand{\mZ}{\mathcal{Z}}

\newcommand{\mX}{\mathcal{X}}

\newcommand{\bp}{\begin{proof}}
\newcommand{\ep}{\end{proof}}

\newcommand{\udim}{\overline{\dim}_B\,}
\newcommand{\ldim}{\underline{\dim}_B\,}

\newcommand{\hdim}{\dim_H}

\newcommand{\lmodbdim}{\underline{\dim}_{\text{\it MB}}\,}
\newcommand{\umodbdim}{\overline{\dim}_{\text{\it MB}}\,}

\DeclareMathOperator{\Id}{Id}

\DeclareMathOperator{\supp}{supp}

\DeclareMathOperator{\Leb}{Leb}

\DeclareMathOperator{\Ker}{Ker}

\DeclareMathOperator{\Per}{Per}
\DeclareMathOperator{\Orb}{Orb}
\DeclareMathOperator{\rank}{rank}
\DeclareMathOperator{\Lin}{Lin}
\DeclareMathOperator{\Gr}{Gr}
\DeclareMathOperator{\Span}{Span}

\makeatother

\providecommand{\conjecturename}{Conjecture}
\providecommand{\corollaryname}{Corollary}
\providecommand{\definitionname}{Definition}
\providecommand{\examplename}{Example}
\providecommand{\lemmaname}{Lemma}
\providecommand{\problemname}{Problem}
\providecommand{\propositionname}{Proposition}
\providecommand{\remarkname}{Remark}
\providecommand{\theoremname}{Theorem}
\providecommand{\taskname}{Task}

\begin{document}

\title{A probabilistic Takens theorem}

\author[K. Bara\'{n}ski]{Krzysztof Bara\'{n}ski$^1$}
\address{$^1$Institute of Mathematics, University of Warsaw, ul. Banacha 2, 02-097 Warszawa, Poland}
\email{baranski@mimuw.edu.pl}

\author[Y. Gutman]{Yonatan Gutman$^2$}
\address{$^2$Institute of Mathematics, Polish Academy of Sciences,
	ul. \'Sniadeckich 8, 00-656 Warszawa, Poland}
\email{y.gutman@impan.pl}

\author[A. \'{S}piewak]{Adam \'{S}piewak$^1$}
\email{a.spiewak@mimuw.edu.pl}

\begin{abstract}
Let $X \subset \R^N$ be a Borel set, $\mu$ a Borel probability measure on $X$ and $T\colon X \to X$ a locally Lipschitz and injective map. Fix $k \in \N$ strictly greater than the (Hausdorff) dimension of $X$ and assume that the set of $p$-periodic points of $T$ has dimension smaller than $p$ for $p=1, \ldots, k-1$. We prove that for a typical polynomial perturbation $\tilde{h}$ of a given locally Lipschitz function $h \colon  X \to \R$, the $k$-delay coordinate map $x \mapsto (\tilde{h}(x), \tilde{h}(Tx), \ldots, \tilde{h}(T^{k-1}x))$ is injective on a set of full $\mu$-measure. This is a probabilistic version of the Takens delay embedding theorem as proven by Sauer, Yorke and Casdagli. We also provide a non-dynamical probabilistic embedding theorem of similar type, which strengthens a previous result by Alberti, B\"{o}lcskei, De Lellis, Koliander and Riegler. In both cases, the key improvements compared to the non-probabilistic counterparts are the reduction of the number of required coordinates from $2\dim X$ to $\dim X$ and using Hausdorff dimension instead of the box-counting one. We present examples showing how the use of the Hausdorff dimension improves the previously obtained results.
\end{abstract}



\maketitle

\section{Introduction}\label{sec:intro}
Consider an experimentalist observing a physical system modeled by a discrete time \emph{dynamical system} $(X,T)$, where $T\colon X \to X$ is the evolution rule and the \emph{phase space} $X$ is a subset of the Euclidean space $\R^N$. It often happens that, for a given point $x \in X$, instead of an actual sequence of $k$ states $x, Tx, \ldots, T^{k-1} x$, the observer's access is limited to the values of $k$ \emph{measurements} $h(x), h(Tx), \ldots, h(T^{k-1} x)$, for a real-valued \emph{observable} $h \colon  X \to \R$. Therefore, it is natural to ask, to what extent the original system can be reconstructed from such sequences of measurements and what is the minimal number $k$, referred to as the number of \emph{delay-coordinates}, required for a reliable reconstruction. These questions have emerged in the physical literature (see e.g.~\cite{PCFS80}) and  inspired a number of mathematical results, known as \emph{Takens-type delay embedding theorems}, stating that the reconstruction of $(X,T)$ is possible for certain observables $h$, as long as the measurements $h(x), h(Tx), \ldots, h(T^{k-1} x)$ are known for \emph{all} $x \in X$ and large enough $k$. 

The possibility of performing measurements at every point of the phase space is clearly unrealistic. However, such an assumption enables one to obtain theoretical results which justify the validity of actual procedures used by experimentalists (see e.g.~\cite{hgls05distinguishing, KY90,sgm90distinguishing,sm90nonlinear}). Note that one cannot expect a reliable reconstruction of the system based on the measurements of a given observable $h$, as it may fail to distinguish the states of the system (e.g. if $h$ is a constant function). It is therefore necessary (and rather realistic) to assume that the experimentalists are able to \emph{perturb} the given observable. The first result obtained in this area is the celebrated Takens delay embedding theorem for smooth systems on manifolds \cite[Theorem 1]{T81}. Due to its strong connections with actual reconstruction procedures used in the natural sciences, Takens theorem has been met with great interest among mathematical physicists (see e.g. \cite{HBS15, SYC91, Voss03}). Let us recall its extension due to Sauer, Yorke and Casdagli \cite{SYC91}. In this setting, the number $k$ of the delay-coordinates should be two times larger than the upper box-counting dimension of the phase space $X$ 
(denoted by $\udim X$; see Section~\ref{sec:prem} for the definition), and the perturbation is a polynomial of degree $2k$. The formulation of the result given here follows \cite{Rob11}.

\begin{thm}[{\cite[Theorem 14.5]{Rob11}}]\label{thm:standard_takens}
Let $X \subset \R^N$ be a compact set and let $T\colon  X \to X$ be Lipschitz and injective. Let $k \in \N$ be such that $k > 2\udim X$ and assume $2\udim(\{ x \in X : T^p x = x \}) < p$ for $p=1, \ldots, k-1$. Let $h \colon  \R^N \to \R$ be a Lipschitz function and $h_1, \ldots, h_m\colon  \R^N \to \R$ a basis of the space of polynomials of degree at most $2k$. For $\alpha = (\alpha_1, \ldots, \alpha_m) \in \R^m$ denote by $h_\alpha \colon  \R^N \to \R$ the map
\[ 
h_\alpha (x) = h(x) + \sum \limits_{j=1}^m \alpha_j h_j (x). 
\]
Then for Lebesgue almost every $\alpha = (\alpha_1, \ldots, \alpha_m) \in \R^m$, the transformation
\[ 
\phi_\alpha^T\colon X \to \R^k, \qquad  \phi_\alpha^T(x) = (h_\alpha(x), h_\alpha(Tx), \ldots, h_\alpha(T^{k-1} x)) 
\]
is injective on $X$.
\end{thm}

The map $\phi_\alpha^T$ is called the \emph{delay-coordinate map}. Note that Theorem~\ref{thm:standard_takens} applies to any compact set $X \subset \R^N$, not necessarily a manifold. This is a useful feature, as it allows one to consider sets with a complicated geometrical structure, such as fractal sets arising as attractors in chaotic dynamical systems, see e.g.~\cite{ER85}. Moreover, the upper box-counting dimension of $X$ can be smaller than the dimension of any smooth manifold containing $X$, so Theorem~\ref{thm:standard_takens} may require fewer delay-coordinates than its smooth counterpart in \cite{T81}.

As it was noted above, usually an experimentalist may perform only a finite number of observations $h(x_j), \ldots, h(T^{k-1}x_{j})$ for some points $x_j \in X$, $j = 1, \ldots, l$. We believe it is realistic to assume there is an (explicit or implicit) \emph{random} process determining which initial states $x_{j}$ are accessible to the experimentalist. In this paper we are interested in the question of reconstruction of the system in presence of such process. Mathematically speaking, this corresponds to fixing a probability measure $\mu$ on $X$ and asking whether the delay-coordinate map $\phi_\alpha^T$ is injective \emph{almost surely} with respect to $\mu$. Since in this setting we are allowed to neglect sets of probability zero, it is reasonable to ask whether the minimal number of delay-coordinates sufficient for the reconstruction of the system can be smaller than $2\dim X$. Our main result states that this is indeed the case, and the number of delay-coordinates can be reduced by half for \emph{any} (Borel) probability measure. 

The problem of determining the minimal number of delay-coordinates required for reconstruction has been already considered in the physical literature. In \cite{PCFS80}, the authors analyzed an algorithm which may by interpreted as an attempt to determine this number in a probabilistic setting. Our work provides rigorous results in this direction. The following theorem is a simplified version of our result.

\begin{thm}[{\bf Probabilistic Takens delay embedding theorem}{}]\label{thm:takens_simple}
Let $X \subset \R^N$ be a Borel set, $\mu$ a Borel probability measure on $X$ and $T\colon  X \to X$ an injective, locally Lipschitz map. Take $k \in \N$ such that $k > \dim X$ and assume that for $p=1, \ldots, k-1$ we have $\dim(\{ x \in X : T^p x = x \}) < p$ or $\mu(\{ x \in X : T^p x = x \}) = 0$. Let $h \colon X \to \R$ be a locally Lipschitz function and $h_1, \ldots, h_m\colon  \R^N \to \R$ a basis of the space of real polynomials of $N$ variables of degree at most $2k-1$. For $\alpha = (\alpha_1, \ldots, \alpha_m) \in \R^m$ denote by $h_\alpha \colon  \R^N \to \R$ the map
	\[ h_\alpha (x) = h(x) + \sum \limits_{j=1}^m \alpha_j h_j (x). \]
	Then for Lebesgue almost every $\alpha = (\alpha_1, \ldots, \alpha_m) \in \R^m$, there exists a Borel set $X_\alpha \subset X$ of full $\mu$-measure, such that the delay-coordinate map
	\[ \phi_\alpha^T\colon X \to \R^k, \qquad  \phi_\alpha^T(x) = (h_\alpha(x), h_\alpha(Tx), \ldots, h_\alpha(T^{k-1} x))  \]
	is injective on $X_\alpha$.
\end{thm}

In the above theorem, the dimension $\dim$ can be chosen to be any of $\hdim, \ldim, \udim$ (Hausdorff, lower and upper box-counting dimension; for the definitions see Section~\ref{sec:prem}). Recall that for any Borel set $X$ one has
\begin{equation}\label{eq:hdim_bdim}
\hdim X \leq \ldim X \leq \udim X
\end{equation} 
(see e.g.~\cite[Proposition~3.4]{falconer2014fractal}). 
Since the inequalities in \eqref{eq:hdim_bdim} may be strict, using the Hausdorff dimension instead of the box-counting one(s) may reduce the required number of delay-coordinates. In particular, there are compact sets $X \subset \R^N$ with $\hdim X=0$ and $\udim X = N$, hence Theorem \ref{thm:takens_simple} can reduce significantly the number of required delay-coordinates compared to Theorem \ref{thm:standard_takens} (in a probabilistic setting).

Notice that in Theorem~\ref{thm:takens_simple} we do not assume that the measure $\mu$ is $T$-invariant. However, the invariance of $\mu$ provides some additional benefits, as shown in the following remark.

\begin{rem}[{\bf Invariant measure case}{}]\label{rem:invariant_simple} Suppose that the measure $\mu$ in Theorem~\ref{thm:takens_simple} is additionally $T$-invariant, i.e.~$\mu(Y) = \mu(T^{-1}(Y))$ for every Borel set $Y \subset X$. Then  
the set $X_{\alpha}$ can be chosen to satisfy $T(X_{\alpha}) = X_{\alpha}$. Moreover, if $\mu$ is $T$-invariant and ergodic (i.e.~$T^{-1}(Y) = Y$ can occur only for sets $Y$ of $0$ or full $\mu$-measure), then the assumption on the periodic points of $T$ in Theorem~\ref{thm:takens_simple} can be omitted.
\end{rem}

Note that in the case when the measure $\mu$ is $T$-invariant, Theorem~\ref{thm:takens_simple} and Remark \ref{rem:invariant_simple} show that for a suitable choice of $X_\alpha$, the map $\phi_{\alpha}^T$ is injective on the invariant set $X_\alpha$, which implies that the dynamical system $(\hat X, \hat T)$ for $\hat X = \phi_{\alpha}^T(X_{\alpha})$, $\hat T = \phi_{\alpha}^T \circ T \circ (\phi_{\alpha}^T)^{-1}$, is a model of the system $(X,T)$ embedded in $\R^k$. 

An extended versions of Theorem~\ref{thm:takens_simple} and Remark~\ref{rem:invariant_simple} are presented and proved in Section~\ref{sec:takens} as Theorem~\ref{thm:takens} and Remark~\ref{rem:invariant}, respectively. Theorem~\ref{thm:takens} shows that the assumption $k > \dim X$ can be slightly weakened, and in addition to locally Lipschitz functions $h$, one can consider locally $\beta$-H\"older functions for suitable $\beta \in (0,1]$. Moreover, one can replace the probabilistic measure $\mu$ by any Borel $\sigma$-finite measure on $X$. For details, see Section~\ref{sec:takens}.

Notice that to eliminate the assumption on the periodic points of $T$ in Theorem~\ref{thm:takens_simple}, one can also consider systems with `few' or no periodic points. For instance, as proved in \cite{y69}, a flow on a subset of Euclidean space given by an autonomous differential equation $\dot{x}=F(x)$, where $F$ is Lipschitz with a constant $L$, has no periodic orbits of period smaller than $\frac{2\pi}{L}$. It follows that if $T$ is a $t$-time map for such a flow with $t < \frac{2\pi}{L \dim X}$, then it has no periodic points of periods smaller than $\dim X$ and therefore the assumption on periodic points in Theorem~\ref{thm:takens_simple} can be omitted (compare also \cite[Remark~1.2]{Gut16}). The same holds if the number of periodic points of a given period is finite, which by the Kupka--Smale theorem is a generic condition in the space of $C^r$-diffeomorphisms ($r\geq 1$) of a compact manifold equipped with the uniform $C^r$-topology\footnote{According to the Kupka--Smale theorem (\cite[Chapter 3, Theorem 3.6]{palis1982geometric}) it is generic that the periodic points are hyperbolic and thus periodic points of a given period are isolated by the Hartman--Grobman theorem (\cite[Chapter 2, Theorem 4.1]{palis1982geometric}).}.

As has been mentioned already, Takens theorems are used in order to justify actual (approximate) delay map procedures based on real experimental data (see e.g. \cite{SugiharaFishes,hgls05distinguishing,sgm90distinguishing,sm90nonlinear}). Note, however, that in the cited papers the dimension of the phase space $X$ is deduced \emph{a posteriori} from the properties of the \emph{time series} (orbits of the delay coordinate map for a given observable). It would be very interesting to know whether in the literature it has been observed for some experimental data originating from a space $X$ with known dimension that it is sufficient to have $k \approx \dim X$ (instead of $k \approx 2 \dim X$) delay-coordinates (in other words, time series of length $k$) in the framework of such procedures.

In this paper we focus our attention to the case when the space $X$ is a subset of a finite-dimensional Euclidean space. Takens-type delay embedding theorems have also been extended to finite-dimensional subsets of Banach spaces (see e.g.~\cite{Rob05}). It is a natural question, whether our probabilistic embedding theorems can also be transferred into the infinite-dimensional setup. This problem will be considered in a subsequent work.

Takens-type delay embedding theorems can be seen as dynamical versions of \emph{embedding theorems} which specify when a finite-dimensional set can be embedded into a Euclidean space. Indeed, under the assumptions of Theorem~\ref{thm:standard_takens}, the delay-coordinate map $\phi_\alpha^T$ is an embedding of $X$ into $\R^k$ for typical $\alpha$. Embedding theorems in various categories have been extensively studied in a number of papers (see Section~\ref{sec:embedding} for a more detailed discussion). Recently, Alberti, B\"{o}lcskei, De Lellis, Koliander and Riegler \cite{Riegler18} proved a probabilistic embedding theorem involving the modified lower box-counting dimension of the set (see Theorem~\ref{thm:riegler}). We are able to improve this result by considering the Hausdorff dimension. Below we present a simplified version of our theorem, which can be seen as a non-dynamical counterpart of Theorem~\ref{thm:takens_simple}. 

\begin{thm}[{\bf Probabilistic embedding theorem}{}]\label{thm:embed_simple}
Let $X \subset \R^N$ be a Borel set and let $\mu$ be a Borel probability measure on $X$. Take $k \in \N$ such that the $k$-th Hausdorff measure of $X$ is zero $($it suffices to take $k > \hdim X)$ and let $\phi \colon X \to \R^k$ be a locally Lipschitz map. Then for Lebesgue almost every linear transformation $L \colon  \R^N \to \R^k$ there exists a Borel set $X_L \subset X$ of full $\mu$-measure, such that $\phi_L = \phi + L$ is injective on $X_L$.
\end{thm}

The extended version of the theorem is formulated and proved in Section~\ref{sec:embedding} as Theorem~\ref{thm:embed}. In particular, we obtain the following geometric corollary (see Section~\ref{sec:embedding} for details).

\begin{cor}[{\bf Probabilistic injective projection theorem}{}]\label{cor:embed-proj_simple}
Let $X \subset \R^N$ be a Borel set and let $\mu$ be a Borel probability measure on $X$. Then for every $k > \hdim X$ and almost every $k$-dimensional linear subspace $S \subset \R^N$, the orthogonal projection of $X$ into $S$ is injective on a full $\mu$-measure subset of $X$.
\end{cor}

Notice that by the Marstrand--Mattila projection theorem (see \cite{Marstrand,Mattila-proj}), if $X \subset \R^N$ is Borel and $k \geq \hdim X$, then for almost all $k$-dimensional linear subspaces $S\subset \R^N$, the image of $X$ under the orthogonal projection into $S$ has Hausdorff dimension equal to $\hdim X$.
Note also that Sauer and Yorke proved in \cite{SauerYorke97} that the dimension\footnote{Any one of the dimensions mentioned above and denoted by $\dim$.} of a bounded Borel subset $X$ of $\R^N$  is preserved under typical smooth maps and typical delay-coordinate maps into $\R^k$ as long as $k \ge \dim X$.

In this paper we also provide several examples. Example \ref{ex:circle} shows that in general the condition $k > \hdim X$ in Theorem \ref{thm:embed_simple} cannot be replaced by $k \geq \hdim X$. Example \ref{ex:no_linear_takens} shows that linear perturbations of the observable are not sufficient for Takens theorem. Section~\ref{sec:examples} contains a pair of examples. The first one is based on Kan's example from the Appendix to \cite{SYC91}, showing that condition $k > 2\hdim X$ is not sufficient for existence of a linear transformation into $\R^k$ which is injective on $X$. As in the probabilistic setting one can work with the Hausdorff dimension, we consider a set $X \subset \R^2$ similar to the one provided by Kan, which cannot be embedded linearly into $\R$, but when endowed with a natural probability measure, almost every linear transformation $L\colon  \R^2 \to \R$ is injective on a set of full measure. The second example provides a probability measure with $\hdim \mu < \underline{\dim}_{\,\it MB\,} \mu$, showing that Theorem \ref{thm:embed_simple} strengthens a previous result from \cite{Riegler18}.

\subsection*{Organization of the paper} The paper is organized as follows. In Section~\ref{sec:prem} we introduce notation, definitions and preliminary results. Section~\ref{sec:embedding} contains the formulation and proof of the extended version of the probabilistic embedding theorem (Theorem~\ref{thm:embed}), while Section~\ref{sec:takens} is devoted to the proof of the extended version of the probabilistic Takens delay embedding theorem (Theorem~\ref{thm:takens}). In Section~\ref{sec:examples} we present examples showing how the use of the Hausdorff dimension improves the previously obtained results.

\subsection*{Acknowledgements} We are grateful to Erwin Riegler for helpful discussions and to the anonymous referees for helpful comments. Y.~G. and A.~\'S. were partially supported by the National Science Centre (Poland) grant 2016/22/E/ST1/00448. 

\section{Preliminaries}\label{sec:prem}
Consider the Euclidean space $\R^N$ for $N \in \N$, with the standard inner product $\langle \cdot, \cdot\rangle$ and the norm $\| \cdot \|$. The open $\delta$-ball around a point $x \in \R^N$ is denoted by $B_N(x, \delta)$. By $|X|$ we denote the diameter of a set $X \subset \R^N$. We say that function $\phi\colon  X \to \R^k,\ X \subset \R^N$ is \emph{locally $\beta$-H\"{o}lder} for $\beta > 0$ if for every $x \in X$ there exists an open set $U \subset \R^N$ containing $x$ such that $\phi$ is $\beta$-H\"{o}lder on $U \cap X$, i.e. there exists $C>0$ such that
\[
\| \phi(x) - \phi(y) \| \leq C\|x - y\|^{\beta}
\]
for every $x,y \in U \cap X$. We say that $\phi$ is \emph{locally Lipschitz} if it is locally $1$-H\"{o}lder. 

For $k \le N$ we write $\Gr(k, N)$ for the $(k, N)$-{\emph{Grassmannian}, i.e.~the space of all $k$-dimensional linear subspaces of $\R^N$, equipped with the standard rotation-invariant (Haar) measure, see \cite[Section 3.9]{mattila} (and \cite{FR02} for another construction of a rotation-invariant measure on the Grassmannian). By $\eta_N$ we denote the normalized Lebesgue measure on the unit ball $B_N(0,1)$, i.e.
\[
\eta_N = \frac{1}{\Leb (B_N(0,1))} \Leb|_{B_N(0,1)},
\]
where $\Leb$ is the Lebesgue measure on $\R^N$. 

For $s>0$, the \emph{$s$-dimensional $($outer$)$ Hausdorff measure} of a set $X \subset \R^N$ is defined  as
\[ \mH^s(X) = \lim \limits_{\delta \to 0}\ \inf \Big\{ \sum \limits_{i = 1}^{\infty} |U_i|^s : X \subset \bigcup \limits_{i=1}^{\infty} U_i,\ |U_i| \leq \delta  \Big\}.\]
The \emph{Hausdorff dimension} of $X$ is given as
\[ \hdim X = \inf \{ s > 0 : \mathcal{H}^s(X) = 0 \} = \sup \{ s > 0 : \mathcal{H}^s(X) = \infty \}. \]
For a bounded set $X \subset \R^N$ and $\delta>0$, let $N(X, \delta)$ denote the minimal number of balls of diameter at most $\delta$ required to cover $X$. The \emph{lower} and \emph{upper box-counting $($Minkowski$)$ dimensions} of $X$ are defined as
\[ \ldim X = \liminf \limits_{\delta \to 0} \frac{\log N(X,\delta)}{-\log \delta}\quad \text{ and }\quad \udim X = \limsup \limits_{\delta \to 0} \frac{\log N(X,\delta)}{-\log \delta}. \]
The lower (resp.~upper) box-counting dimension of an unbounded set is defined as the supremum of the lower (resp.~upper) box-counting dimensions of its bounded subsets.

The \emph{lower} and \emph{upper modified box-counting dimensions} of $X \subset \R^N$ are defined as
\begin{align*}
\lmodbdim X &= \inf \Big\{ \sup \limits_{i \in \N} \ldim K_i : X \subset  \bigcup \limits_{i=1}^{\infty} K_i,\ K_i \text{ compact} \Big\},\\
\umodbdim X &= \inf \Big\{ \sup \limits_{i \in \N} \udim K_i : X \subset \bigcup \limits_{i=1}^{\infty} K_i,\ K_i \text{ compact} \Big\}.
\end{align*}
With this notation, the following inequalities hold:
\begin{equation}\label{eq:hdim_mbdim_bdim}
\begin{aligned}
&\hdim X  \leq \lmodbdim X \leq \umodbdim X \leq \udim X,\\
&\hdim X \leq \lmodbdim X \leq \ldim X \leq \udim X.
\end{aligned}
\end{equation}
We define dimension of a finite Borel measure $\mu$ in $\R^N$ as
\[
\dim \mu = \inf\{ \dim X: X \subset \R^N \text{ is a Borel set of full $\mu$-measure} \}.
\]
Here $\dim$ may denote any one of the dimensions defined above. Recall that for a measure $\mu$ on a set $X$ and a measurable $Y \subset X$ we say that $Y$ is of \emph{full $\mu$-measure}, if $\mu(X\setminus Y) = 0$.
For more information on dimension theory in Euclidean space see \cite{falconer2014fractal, mattila, Rob11}.

For $N,k \in \N$ let $\Lin(\mathbb R^N; \mathbb R^k)$ be the space of all linear transformations $L \colon  \R^N \to \R^k$. Such transformations are given by 
\begin{equation}\label{eq:L}
L x = \big(\langle l_1, x \rangle , \ldots, \langle l_k, x \rangle\big),
\end{equation}
where $l_1, \ldots, l_k \in \R^N$. Thus, the space $\Lin(\R^N; \R^k)$ can be identified with $(\R^N)^k$, and the Lebesgue measure on $\Lin(\R^N; \R^k)$ is understood as $\bigotimes \limits_{j=1}^k \Leb$, where $\Leb$ is the Lebesgue measure in $\R^N$. Within the space $\Lin(\R^N; \R^k)$ we consider the space $E^N_k$ consisting of all linear transformations $L \colon  \R^N \to \R^k$ of the form \eqref{eq:L}, for which $l_1, \ldots, l_k \in B_N(0, 1)$. Note that by the Cauchy-Schwarz inequality,
\begin{equation}\label{eq:E}
\|Lx\| \leq \sqrt{N} \,\|x\|
\end{equation}
for every $L \in E^N_k$ and $x \in \R^N$.

By $\eta_{N,k}$ we denote the normalized Lebesgue measure on $E^N_k$, i.e.~the probability measure on $E^N_k$ given by 
\[
\eta_{N,k} = \bigotimes \limits_{j=1}^k \frac{1}{\Leb(B_N(0, 1))} \Leb|_{B_N(0, 1)}.
\]
The following geometrical inequality, used in \cite{HK99} (see also \cite[Lemma 4.1]{Rob11}) is the key ingredient of the proof of Theorem~\ref{thm:embed}.

\begin{lem}\label{lem: key_ineq_linear}
Let $L \colon  \R^N \to \R^k$ be a linear transformation. Then for every $x \in \R^N \setminus \{ 0 \}$, $z \in \R^k$ and $\eps>0$,
	\[
	\eta_{N,k} (\{ L \in E^N_k : \|Lx + z \| \leq \eps \}) \leq CN^{k/2}\frac{\eps^k}{\|x\|^k}, \]
	where $C > 0$ is an absolute constant.
\end{lem}

For $L \in \Lin(\mathbb R^m; \mathbb R^k)$, where $m,k \in \N$, denote by $\sigma_p(L)$, $p \in \{1, \ldots, k\}$, the $p$-th largest \emph{singular value} of the matrix $L$, i.e.~the $p$-th largest square root of an eigenvalue of the matrix $L^*L$. In the proof of Theorem~\ref{thm:takens}, instead of Lemma~\ref{lem: key_ineq_linear} we will use the following lemma, proved as \cite[Lemma~4.2]{SYC91} (see also \cite[Lemma 14.3]{Rob11}).

\begin{lem}\label{lem: key_ineq_inter}
Let $L \colon  \R^m \to \R^k$ be a linear transformation. Assume that $\sigma_p(L) > 0$ for some $p \in \{1, \ldots, k\}$. Then for every $z \in \R^k$ and $\rho, \eps > 0$,
\[ \frac{\Leb(\{ \alpha \in B_m(0, \rho) : \|L \alpha + z \| \leq \eps \})}{\Leb (B_m(0, \rho))} \leq C_{m,k} \Big(\frac{\eps}{\sigma_p(L) \, \rho}\Big)^p, \]
where $C_{m,k} > 0$ is a constant depending only on $m,k$ and $\Leb$ is the Lebesgue measure on $\R^m$.
\end{lem}

To verify the measurability of the sets occurring in subsequent proofs, we will use the two following elementary lemmas. A measure $\mu$ on a set $X$ is called \emph{$\sigma$-finite} if there exists a countable collection of measurable sets $A_n,\ n \in \N$ such that $\mu(A_n)<\infty$ for each $n \in \N$ and $\bigcup \limits_{n=1}^{\infty}A_n = X$. Recall that a \emph{$\sigma$-compact} set is a countable union of compact sets.

\begin{lem}\label{lem:dimh_fsigma}
	Let $X\subset \R^N$ be a Borel set and let $\mu$ be a Borel $\sigma$-finite  measure on $X$. Then there exists a $\sigma$-compact set $K \subset X$ of full $\mu$-measure.
\end{lem}
\begin{proof} Follows directly from the fact that a $\sigma$-finite Borel measure in a Euclidean space is regular (see e.g.~\cite[Theorem 1.1]{Billingsley99}).
\end{proof}

\begin{lem}\label{lem:measurability}
	Let $\mX, \mZ$ be metric spaces. Then the following hold. 
	\begin{itemize}
		\item[(a)] If $K \subset \mX \times \mZ$ is $\sigma$-compact, then so is $\pi_\mX(K)$, where $\pi_\mX\colon  \mX \times \mZ \to \mX$ is the projection given by $\pi_\mX(x,z)=x$. In particular, $\pi_\mX(K)$ is Borel.
		\item[(b)] If $\mX$ is $\sigma$-compact, $F\colon  \mX \to \mZ$ is continuous and $K \subset \mZ$ is $\sigma$-compact, then $F^{-1}(K)$ is $\sigma$-compact, hence Borel.
		\item[(c)]
		If $\mX,\mZ$ are $\sigma$-compact, $F\colon  \mX \times \mZ \to \R^k$, $k \in \N$, is continuous and $K \subset \mX$ is $\sigma$-compact, then the set 
		\[ \{ (x, z) \in \mX \times \mZ : F(x,z) = F(y,z) \text{ for some } y \in K \setminus\{ x \} \} \]
		is $\sigma$-compact and hence Borel.
	\end{itemize}
\end{lem}
\begin{proof}
	The statement (a) follows from the fact that $\pi_\mX$ is continuous, and a continuous image of a compact set is also compact. To show (b),
	it is enough to notice that $F^{-1}(K)$ is a countable union of closed subsets of a $\sigma$-compact space. To check (c), let $\pi_{\mX \times \mZ}\colon  \mX \times K \times \mZ \to \mX \times \mZ$ be the projection $\pi_{\mX \times \mZ}(x,y,z) = (x,z)$. Then
	\begin{align*}
	&\{ (x, z) \in \mX \times \mZ : F(x,z) = F(y,z) \text{ for some } y \in K \setminus\{ x \} \} \\
		&= \pi_{\mX \times \mZ}\big(\{ (x, y, z) \in \mX \times K \times \mZ : F(x,z) = F(y,z), \: d(x,y) \neq 0 \}\big) \\
		& = \bigcup \limits_{n = 1}^{\infty} \pi_{\mX \times \mZ}\big(\{ (x, y, z) \in \mX \times K \times \mZ : F(x,z) = F(y,z), \: d(x,y) \geq \frac{1}{n} \}\big),
	\end{align*}
	where $d$ is the metric in $\mX$. Since $d$ is continuous, we can use (a) and (b) to end the proof.
\end{proof}

\section{Probabilistic embedding theorem}\label{sec:embedding}

In this section we prove an extended version of the probabilistic embedding theorem, formulated below. Obviously, Theorem~\ref{thm:embed_simple} follows from Theorem \ref{thm:embed}

\begin{thm}[{\bf Probabilistic embedding theorem -- extended version}{}] \label{thm:embed}
Let $X \subset \R^N$ be a Borel set and $\mu$ be a Borel $\sigma$-finite measure on $X$. Take $k \in \N$ and $\beta \in (0,1]$ such that $\mH^{\beta k}(X) = 0$ and let $\phi \colon  X \to \R^k$ be a locally $\beta$-H\"older map. Then for Lebesgue almost every linear transformation $L\colon  \R^N \to \R^k$ there exists a Borel set $X_L \subset X$ of full $\mu$-measure, such that the map $\phi_L = \phi + L$ is injective on $X_L$.
\end{thm}

\begin{rem} It is straightforward to notice that if $\dim_H X = 0$, then $\phi$ can be taken to be an arbitrary H\"older map. 
\end{rem}

\begin{proof}[Proof of Theorem~\rm\ref{thm:embed}]

Note first that it sufficient to prove that the set $X_L$ exists for $\eta_{N,k}$-almost every $L \in E^N_k$. Indeed, if this is shown, then for a given locally $\beta$-H\"older map $\phi \colon  X \to \R^k$ we can take sets $\mathcal L_j \subset E^N_k$, $j \in \N$, such that $\eta_{N,k}(\mathcal L_j) = 1$ and for every $\tilde L \in \mathcal L_j$ the map $(\phi/j)_{\tilde L} = \phi/j + \tilde L$ is injective on a Borel set $X_{\tilde L}^{(j)} \subset X$ of full $\mu$-measure. Then the set $\mathcal L = \bigcup_{j\in \N} \{j\tilde L: \tilde L \in \mathcal L_j\} \subset \Lin(\mathbb R^N; \mathbb R^k)$ has full Lebesgue measure and for every $L \in \mathcal L$ there exists $j$ such that $L/j \in \mathcal L_j$, so $(\phi/j)_{L/j} = (\phi + L)/j$ (and hence $\phi_L$) is injective on $X_L = \bigcap_{j\in\N} X_{L/j}^{(j)}$, which has full $\mu$-measure. 

By Lemma~\ref{lem:dimh_fsigma}, we can assume that $X$ is $\sigma$-compact. Take $k \in \N$, $\beta \in (0,1]$ with $\mH^{\beta k}(X) = 0$ and a locally $\beta$-H\"older map $\phi \colon X \to \R^k$. Set
\[ A = \{ (x, L) \in X \times E^N_k : \phi_L (x) = \phi_L (y) \text{ for some } y \in X \setminus \{ x \}\}. \]
By Lemma~\ref{lem:measurability}, $A$ is Borel.
For $x \in X$ and $L \in E^N_k$, denote by $A_x$ and $A^L$, respectively, the sections
\[ A_x = \{ L \in E^N_k : (x, L) \in A \},\quad A^L = \{ x \in X : (x, L) \in A \}. \]
The sets $A_x$ and $A^L$ are Borel as sections of a Borel set. Observe first that in order to prove the theorem it is enough to show $\eta_{N,k}(A_x) = 0$ for every $x \in X$, since then by Fubini's theorem (\cite[Thm. 8.8]{R87}), $(\eta_{N,k} \otimes \mu) (A) = 0$ and, consequently, $\mu(A^L) = 0$ for $\eta_{N,k}$-almost every $L \in E_k^N$. Since $\phi_L$ is injective on $X \setminus A^L$, the assertion of the theorem is true.

Take a point $x \in X$. Since $\phi$ is locally $\beta$-H\"older and $X$ is separable, there exists a countable covering of $X$ by open sets $U_j \subset \R^N$, $j \in \N$, such that 
\begin{equation}\label{eq:holder}
\|\phi(y) - \phi(y')\| \le C_j \|y - y'\|^\beta \qquad \text{for every} \quad y,y' \in U_j \cap X 
\end{equation}
for some $C_j > 0$. Let 
\[
K_{n} = \Big\{ y \in X : \frac{1}{n} \leq \|x-y\| \Big\}.
\]
To show $\eta_{N,k}(A_x) = 0$, it suffices to prove $\eta_{N,k}(A_{x, j, n}) = 0$ for every $j, n \in \N$, where
\[
A_{x, j, n} = \{ L \in E^N_k : \phi_L (x) = \phi_L (y) \text{ for some } y \in U_j \cap K_{n}\}.
\]
Note that by Lemma~\ref{lem:measurability}, the set $A_{x, j, n}$ is Borel.

Take $j, n \in \N$ and fix a small $\eps > 0$. Since $ \mH^{\beta k}(U_j \cap K_{n}) \leq \mH^{\beta k} (X) = 0$, there exists a collection of balls $B_N(y_i, \eps_i)$, $i \in \N$, for some $y_i \in U_j \cap K_{n}$ and $\eps_i > 0$, such that
\begin{equation}\label{e:cover} U_j \cap K_{n} \subset \bigcup \limits_{i \in \N} B_N(y_i, \eps_i)\quad \text{and} \quad  \sum \limits_{i = 1}^\infty \eps_i^{\beta k} \leq \eps.
\end{equation}
Take $L \in A_{x, j, n}$ and $y \in U_j \cap K_{n}$ such that $\phi_L (x) = \phi_L (y)$. Then $y \in B_N (y_i, \eps_i)$ for some $i \in \N$ and
\begin{align*}
\| L(y_i - x) + \phi(y_i) - \phi(x) \| &= \| \phi_L(y_i) - \phi_L(x)\|\\
&= \| \phi_L(y_i) - \phi_L(y)\|\\
& \leq \|\phi(y_i) - \phi(y)\| + \| L(y_i-y)\|\\
&\leq C_j \|y_i-y\|^\beta + \sqrt{N} \|y_i - y\|\\
&\leq M_j \eps_i^\beta
\end{align*}
for some $M_j >0$, by \eqref{eq:E} and \eqref{eq:holder}. This shows that
\[ A_{x, j, n} \subset \bigcup \limits_{i \in \N} \{ L \in E^N_k : \| L(y_i - x) + \phi(y_i) - \phi(x) \| \leq M_j\eps_i^\beta \}. \]
By Lemma~\ref{lem: key_ineq_linear}, (\ref{e:cover}) and the fact $y_i \in K_n$, we have
\begin{align*}
	\eta_{N,k}(A_{x, j, n}) &\leq \sum \limits_{i = 1}^\infty \eta_{N,k}(\{ L \in E^N_k : \| L(y_i - x) + \phi(y_i) - \phi(x) \| \leq M_j\eps_i^\beta \})\\
	&\leq \frac{CN^{k/2}M_j^k}{1/n^k}\sum \limits_{i = 1}^\infty \eps_i^{\beta k} \leq  CN^{k/2}M_j^k n^k\eps.
\end{align*}
Since  $\eps > 0$ was arbitrary, we obtain $\eta_{N,k}(A_{x, j, n}) = 0$, which ends the proof.
\end{proof}

\begin{rem}
	Note that the assumption $\mH^{\beta k}(X) = 0$ is fulfilled if $\hdim X < \beta k$, so Theorem~\ref{thm:embed} is indeed a Hausdorff dimension embedding theorem. Moreover, it may happen that $\mH^{\beta k}(X) = 0$ and $\hdim X = \beta k$.
\end{rem}

As a simple consequence of Theorem~\ref{thm:embed}, we obtain the following corollary, formulated in a slightly simplified version in Section~\ref{sec:intro} as Corollary~\ref{cor:embed-proj_simple}. 

\begin{cor}[{\bf Probabilistic injective projection theorem -- extended version}{}]\label{cor:embed-proj}
Let $X \subset \R^N$ be a Borel set and let $\mu$ be a Borel $\sigma$-finite measure on $X$. Then for every $k \in \N,\ k \leq N$ such that $\mH^k(X) =0$ and almost every $k$-dimensional linear subspace $S \subset \R^N$ $($with respect to the standard measure on the Grassmannian $\Gr(k, N))$, the orthogonal projection of $X$ into $S$ is injective on a full $\mu$-measure subset of $X$ $($depending on $S)$. 
\end{cor}

\begin{proof}[Proof of Corollary~\rm\ref{cor:embed-proj}]
Apply Theorem~\ref{thm:embed} for the map $\phi \equiv 0$. Then we know that a linear map $L \in \Lin(\R^N; \R^k)$ of the form \eqref{eq:L} is injective on a set $X_L \subset X$ of full $\mu$-measure for Lebesgue almost every $(l_1, \ldots, l_k) \in (\R^N)^k$. We can assume that $l_1, \ldots, l_k$ are linearly independent for all such $L$, which also implies that the same holds for $Ll_1, \ldots, Ll_k$. Setting 
\[
S_L = \Span (l_1, \ldots, l_k)
\]
and taking $V_L \in \Lin(\R^k; \R^N)$ defined by $V_L(Ll_j) = l_j$ for $j = 1, \ldots, k$, we have
\[
V_L \circ L = \Pi_{S_L},
\]
where $\Pi_{S_L}$ is the orthogonal projection from $\R^N$ onto $S_L$ and $V_L$ is injective. It follows that $\Pi_{S_L}$ is injective on $X_L$ for almost every $(l_1, \ldots, l_k)$, so $\Pi_{S}$ is injective on a full $\mu$-measure subset of $X$ for almost every $k$-dimensional linear subspace $S \subset \R^N$. 
\end{proof}

Let us note that in general, the requirement $\mH^{\beta k}(X) = 0$ in Theorem \ref{thm:embed} cannot be replaced by the weaker condition $\hdim(X) \leq \beta k$.

\begin{examplex}\label{ex:circle}
Let $k=\beta=1$, $X = \mathbb{S}^1 \subset \R^2$ be the unit circle and let $\mu$ be the normalized Lebesgue measure on $\mathbb{S}^1$. We shall prove that there is no Lipschitz transformation $\phi \colon  \mathbb{S}^1 \to \R$ which is injective on a set of full $\mu$-measure. Let $\phi$ be such a transformation. Then $\phi(\mathbb{S}^1)=[a, b]$ for some compact interval. As $\phi$ is injective on a set of full measure, the interval $[a,b]$ is non-degenerate, i.e. $a < b$. Fix points $x, y \in \mathbb{S}^1$ with $\phi(x) = a, \phi(y) = b$. As $x \neq y$, there are exactly two open arcs $I, J \subset \mathbb{S}^1$ of positive measure joining $x$ and $y$ such that $\overline{I} \cap \overline{J} = \{ x,y\}$ and $\overline{I} \cup \overline{J} = \mathbb{S}^1$. Clearly $\phi(\overline{I}) = \phi(\overline{J}) = [a,b]$. Let $A \subset \mathbb{S}^1$ be a Borel set such that $\phi$ is injective on $A$ and $\mu(A)$=1. As Lipschitz maps transform sets of zero Lebesgue measure to sets of zero Lebesgue measure, we conclude that $\phi(I \cap A)$ and $\phi(J \cap A)$ are disjoint Lebesgue measurable subsets of $[a,b]$ with Lebesgue measure equal to $b-a$. This contradiction shows that no Lipschitz transformation $\phi \colon  \mathbb{S}^1 \to \R$ is injecitive on a full measure set. 
\end{examplex}

Theorem~\ref{thm:embed} strengthens the following embedding theorem, proved recently by Alberti, B\"{o}lcskei, De Lellis, Koliander and Riegler in \cite{Riegler18}. 

\begin{thm}[{\cite[Theorem II.1]{Riegler18}}]\label{thm:riegler}
	Let $\mu$ be a Borel probability measure in $\R^N$ and let $k \in \N$ be such that $k> \lmodbdim \mu$. Then for Lebesgue almost every linear transformation $L \colon  \R^N \to \R^k$ there exists a Borel set $X_L \subset \R^N$ such that $\mu(X_L)=1$ and $L$ is injective on $X_L$.
\end{thm}

In fact, in \cite{Riegler18} the authors introduced the notion of $\lmodbdim \mu$, denoting it by $K(\mu)$ and calling it the \emph{description complexity} of the measure. In particular, Theorem~\ref{thm:riegler} holds for measures $\mu$ supported on a Borel set $X\subset \R^N$ with $\ldim X <k$. By \eqref{eq:hdim_mbdim_bdim}, we have $\hdim \mu \leq \lmodbdim \mu $, and in Section~\ref{sec:examples} we present an example (Theorem~\ref{thm:hdim<lmodbdim}) showing that the inequality may be strict. Therefore, Theorem~\ref{thm:embed} actually strengthens Theorem~\ref{thm:riegler}.

Non-probabilistic embedding theorems were first obtained in topological and smooth categories. The well-known Menger--N\"{o}beling embedding theorem (see e.g.~\cite[Theorem~V.2]{HW41}) states that for a compact metric space $X$ with Lebesgue covering dimension at most $k$, a generic continuous transformation $\phi \colon  X \to \R^{2k+1}$ is injective (and hence defines a homeomorphism between $X$ and $\phi(X)$). Genericity means here that the set of injective transformations $\phi \colon  X \to \R^{2k+1}$ is a dense $G_{\delta}$ subset of $C(X ; \R^{2k+1})$ endowed with the supremum metric. The dimension $2k+1$ is known to be optimal. 
The corresponding result in the category of smooth manifolds is the Whitney embedding theorem (see \cite{Whitney36}). It states that for a given $k$-dimensional $C^r$-manifold $M$, a generic $C^r$-transformation from $M$ to $\R^{2k+1}$ is a $C^r$-embedding (i.e.~an injective immersion of class $C^r$). 

Let us now compare Theorem~\ref{thm:embed} to non-probabilistic embedding theorems involving the box-counting dimension. One of the first results in this area was a theorem by Ma\~{n}\'{e} \cite[Lemma 1.1]{Mane81}. We present its formulation following \cite[Theorem 4.6]{SYC91} and \cite[Theorem 6.2]{Rob11} (originally, Ma\~{n}\'{e} proved that topologically generic linear transformation is injective on $X$).
\begin{thm}\label{thm:standard_embed}
	Let $X \subset \R^N$ be a compact set. Let $k \in N$ be such that $k>2\udim{X}$ $($it suffices to take $k > \hdim(X-X))$. Then Lebesgue almost every linear transformation $L\colon \R^N \to \R^k$ is injective on $X$.
\end{thm}
\begin{rem}\label{rem:hdim_does_not_work}
	As noticed by Ma\~{n}\'{e} and communicated in \cite[p.~627]{ER85}, his original statement in \cite{Mane81} is incorrect. Namely, he assumed $k>2\hdim X+1$ instead of $k > \hdim(X-X)$. However, this is known to be insufficient for the existence of a linear embedding of $X$ into $\R^k$. In fact, in \cite[Appendix A]{SYC91}, Kan presented an example of a set $X \subset \R^m$ with $\hdim X =0$, such that any linear transformation $L \colon  \R^m \to \R^{m-1}$ fails to be injective on $X$. It turns out that the assumption $k > 2\hdim X$ is insufficient, while $k > 2\udim X$ is sufficient. This stems from the fact that the proof of Theorem~\ref{thm:standard_embed} actually requires the property $k > \hdim(X-X)$, and the upper box-counting dimension satisfies
	\begin{equation}\label{e:udim_product}
	\udim(A\times B) \leq \udim(A) + \udim(B),
	\end{equation}
	for $A, B \subset \R^N$, hence
	\[\hdim(X - X) \leq \hdim(X \times X) \leq \udim(X \times X) \leq 2\udim X\]
	(note that this calculation shows that $k>2\udim{X}$  is a stronger assumption than $k > \hdim(X-X)$). On the other hand, (\ref{e:udim_product}) does not hold for the Hausdorff dimension (nor for the lower box-counting dimension), and $\hdim X$ does not control $\hdim(X-X)$. The fact that in Theorem~\ref{thm:embed} we can work with the Hausdorff dimension comes from the application of Fubini's theorem, which enables us to consider covers of the set $X$ itself, instead of $X-X$. In Section \ref{sec:examples} we analyze Kan's example from the point of view of Theorem~\ref{thm:embed}.
\end{rem}
Theorem~\ref{thm:standard_embed} is also true for subsets of an arbitrary Banach space $\mathfrak{B}$ for a prevalent set of linear transformations $L\colon \mathfrak{B} \to \R^k$ (see \cite[Chapter 6]{Rob11} for details).

Note that the linear embedding from Theorem~\ref{thm:embed} need not preserve the dimension of $X$. Indeed, the Hausdorff and box-counting dimensions are invariants for bi-Lipschitz transformations, yet inverse of a linear map on a compact set does not have to be Lipschitz. Therefore, we only know that $\dim \phi_L(X) \leq \dim X$ (see \cite[Proposition 2.8.iv and Lemma 3.3.iv]{Rob11}) and the inequality can be strict. For example, let $\phi \equiv 0$ and $X = \{ (x, f(x)) : x \in [0,1] \}$ be a graph of a (H\"{o}lder continuous) function $f\colon [0,1]\to \R$ with $\hdim X >1$, e.g.~the Weierstrass non-differentiable function. Then the linear projection $L \colon  \R^2 \to \R$ given by $L(x,y)=x$ satisfies $1= \dim L(X) < \hdim X$. The following theorem shows that in the non-probabilistic setting, one can obtain $\beta$-H\"{o}lder continuity of the inverse map for small enough $\beta \in (0,1)$ (see \cite{BAEFN93, EFNT94, HK99} and \cite[Chapter 4]{Rob11}).
\begin{thm}
	Let $X \subset \R^N$ be a compact set. Let $k \in \N$ be such that $k>2\udim{X}$ and let $\beta$ be such that $0 < \beta < 1 - 2\udim X/k$. Then Lebesgue almost every linear transformation $L \colon  \R^N \to \R^k$ is injective on $X$ with $\beta$-H\"{o}lder continuous inverse.
\end{thm}
However, this is not true in the case of Theorem~\ref{thm:embed}.

\begin{rem}\label{rem:no_holder_inverse}
	In general, we cannot claim that the injective map $\phi_L|_{X_L}$ from Theorem~\ref{thm:embed} has a H\"{o}lder continuous inverse. Indeed, it is well-known that for $n \in \N$ there are examples of compact sets $X \subset \R^N$ of Hausdorff and topological dimension equal to $n$, which do not embed topologically into $\R^k$ for $k \leq 2n$ (showing the optimality of the bounds in the Menger--N\"{o}beling embedding theorem, see \cite[Example V.3]{HW41}). Consider a probability measure $\mu$ on $X$ with $\supp \mu = X$, where $\supp$ denotes the topological support of the measure (the intersection of all closed sets of full measure). It is known that such measure exists for any compact set. If the map $\phi_L|_{X_L}$ from Theorem~\ref{thm:embed} for $k = n+1$ had a H\"{o}lder continuous inverse $f = \phi_L^{-1}$, then we could extend $f$ from $\phi_L(X_L)$ to $\R^{n+1}$ preserving the H\"{o}lder continuity (\cite[Theorem IV.7.5]{Banach51}, see also \cite{minty1970extension}). Then $Y = \{ x \in X : f \circ \phi_L(x) = x \}$ would be a closed subset of $X$ with $\mu(Y)=1$, hence $Y=X$, so $\phi_L$ would be homeomorphism between $X$ and $\phi_L(X) \subset \R^{n+1}$, which would give a contradiction.
\end{rem}

\section{Probabilistic Takens delay embedding theorem}\label{sec:takens}
In this section we present the proof of the extended probabilistic Takens delay embedding theorem. It turns out that linear perturbations are insufficient for Takens-type theorems, see Example~\ref{ex:no_linear_takens}. As observed in \cite{SYC91}, it is enough to take perturbations over the space of polynomials of degree $2k$. This can be easily extended to more general families of functions.
\begin{defn}
Let $X$ be a subset of $\R^N$. A family of transformations $h_1, \ldots, h_m \colon  X \to \R$ is called a \emph{$k$-interpolating family} on set $X$, if for every collection of distinct points $x_1, \ldots, x_k \in X$ and every $\xi = (\xi_1, \ldots, \xi_k) \in \R^k$ there exists $(\alpha_1, \ldots, \alpha_m) \in \R^m$ such that 
$\alpha_1 h_1(x_i) + \cdots + \alpha_m h_m(x_i) = \xi_i$
for each $i=1, \ldots, k$. In other words, the matrix
\[ \begin{bmatrix} h_1(x_1) & \ldots & h_m(x_1) \\
\vdots &\ddots & \vdots \\
h_1(x_k) & \ldots & h_m(x_k)
\end{bmatrix} \]
has full row rank as a transformation from $\R^m$ to $\R^k$. Note that the same is true for any collection of $l$ distinct points with $l \leq k$.
\end{defn}
\begin{rem}\label{rem:poly_interp} It is known that any linear basis $h_1, \ldots, h_m$ of the space of real polynomials of $N$ variables of degree at most $k-1$ is a $k$-interpolating family (see e.g.~\cite[Section~1.2, eq.~(1.9)]{poly-interpolation}).
\end{rem} 

For a transformation $T \colon  X \to X$ and $p \in \N$ denote by $\Per_p(T)$ the set of periodic points of minimal period $p$, i.e.
\[ \Per_p(T) = \{ x \in X : T^p x = x  \text{ and } T^j x\neq x \text{ for } j = 1, \ldots, p-1\}.\] 
Let $\mu$ and $\nu$ be measures on a measurable space $(\mX, \mF)$. The measure $\mu$ is called \emph{singular} with respect to $\nu$, if there exists a measurable set $Y \subset \mX$ such that $\mu(\mX \setminus Y) = \nu(Y) = 0$. In this case we write $\mu \perp \nu$. By $\mu|_A$ we denote the restriction of $\mu$ to a set $A \in \mF$.

\begin{thm}[{\bf Probabilistic Takens delay embedding theorem -- extended version}{}]\label{thm:takens}
Let $X \subset \R^N$ be a Borel set, $\mu$ be a Borel $\sigma$-finite measure on $X$ and $T\colon  X \to X$ an injective, locally Lipschitz map. Take $k \in \N$ and $\beta \in (0,1]$ such that $\mH^{\beta k}(X)= 0$ and assume $\mu|_{\Per_p(T)} \perp \mathcal{H}^{\beta p}$ for every $p= 1, \ldots, k-1$. Let $h \colon  X \to \R$ be a locally $\beta$-H\"older function and $h_1, \ldots, h_m\colon  X \to \R$ a $2k$-interpolating family on $X$ consisting of locally $\beta$-H\"{o}lder functions. For $\alpha = (\alpha_1, \ldots, \alpha_m) \in \R^m$ denote by $h_\alpha \colon  X \to \R$ the transformation
\[ h_\alpha (x) = h(x) + \sum \limits_{j=1}^m \alpha_j h_j (x). \]
Then for Lebesgue almost every $\alpha = (\alpha_1, \ldots, \alpha_m) \in \R^m$, there exists a Borel set $X_\alpha \subset X$ of full $\mu$-measure, such that the delay-coordinate map
\[
\phi_\alpha^T\colon  X \to \R^k, \qquad \phi_\alpha^T(x) = (h_\alpha(x), h_\alpha(Tx), \ldots, h_\alpha(T^{k-1} x))
\]
is injective on $X_\alpha$.
\end{thm}

Notice that  Theorem~\ref{thm:takens_simple} follows from Theorem~\ref{thm:takens} by Remark~\ref{rem:poly_interp}.

\begin{rem}[{\bf Invariant measure case -- extended version}{}]\label{rem:invariant} Under the assumptions of Theorem~\ref{thm:takens}, the following hold.
\begin{itemize}
\item[(a)] If the measure $\mu$ is $T$-invariant, then the set $X_{\alpha}$ can be chosen to satisfy $T(X_{\alpha}) \subset X_{\alpha}$.
\item[(b)] If the measure $\mu$ is finite and $T$-invariant, then the set $X_{\alpha}$ can be chosen to satisfy $T(X_{\alpha}) = X_{\alpha}$.
\item[(c)] If the measure $\mu$ is $T$-invariant and ergodic, then the assumption on the periodic points of $T$ in Theorem~\ref{thm:takens} can be omitted.
\end{itemize}
\end{rem}

Under the notation of Theorem~\ref{thm:takens}, we first show a preliminary lemma. For $x \in X$ define its \emph{full orbit} $\Orb(x)$ as
\[\Orb(x) = \{ T^n x : n \ge 0\} \cup \{y \in X : T^n y = x \text{ for some } n \in \N\}.\]
Note that since $T$ is injective, all full orbits are at most countable, and any two full orbits $\Orb(x)$ and $\Orb(y)$ are either equal or disjoint. For $x,y \in X$ let $D_{x,y}$ be the $k\times m$ matrix defined by
\[ D_{x,y} = \begin{bmatrix} h_1(x) - h_1(y) & \ldots & h_m(x) - h_m(y) \\
h_1(Tx) - h_1(Ty) & \ldots & h_m(Tx) - h_m(Ty) \\
\vdots & \ddots & \vdots \\
h_1(T^{k-1}x) - h_1(T^{k-1}y) & \ldots & h_m(T^{k-1}x) - h_m(T^{k-1}y) \\
\end{bmatrix}. \]

\begin{lem}\label{lem:claim}
For $x, y \in X$, the following statements hold. 
\begin{enumerate}[\rm (i)]
\item If $y \neq x$, then $\rank D_{x,y} \geq 1$.
\item If $y \notin \Orb(x)$ and $y \in \Per_p(T)$ for some $p \in \{1, \ldots, k-1\}$, then $\rank D_{x,y} \geq p$.
\item If $y \notin \Orb(x)$ and $y \notin \bigcup \limits_{p=1}^{k-1} \Per_p(T)$, then $\rank D_{x,y} = k$.
\end{enumerate}
\end{lem}
\begin{proof}
For (i), it suffices to observe that the first row of $D_{x,y}$ is non-zero as long as $x\neq y$ and therefore $\rank(D_{x,y}) \geq 1$. Indeed, otherwise we would have $h_j(x) = h_j(y)$ for $j=1, \ldots, m$ which contradicts the fact that $h_1, \ldots, h_m$ is an interpolating family.

Assume now $y \notin \Orb(x)$, which implies $\Orb(y) \cap \Orb(x) = \emptyset$. Let $q$ (resp.~$r$) be a maximal number from $\{1, \ldots, k\}$ such that the points $x, Tx, \ldots, T^{q-1} x$ (resp.~$y, Ty, \ldots, T^{r-1} y$) are distinct. Notice that if $y \in \Per_p(T)$ for some $p \in \{1, \ldots, k-1\}$, then $r = p$, and if $y \notin \bigcup \limits_{p=1}^{k-1} \Per_p(T)$, then $r = k$. Thus, the assertions (ii)--(iii) of the lemma can be written simply as one condition 
\begin{equation}\label{eq:r}
\rank D_{x,y} \geq r. 
\end{equation}
To show that \eqref{eq:r} holds, denote the points $x, Tx, \ldots, T^{q-1} x, y, Ty, \ldots, T^{r-1}y$, preserving the order, by $z_1, \ldots, z_l$, for $l = q + r$. By the definition of $q, r$, we have $1 \leq l \le 2k$ and the points $z_1, \ldots, z_l$ are distinct. Thus, the matrix $D_{x,y}$ can be written as the product
\[ D_{x,y} = J_{x,y} V_{x,y}, \]
where 
\[ V_{x,y}
= \begin{bmatrix} h_1(z_1) & \ldots & h_m(z_1) \\
\vdots & \ddots & \vdots \\
h_1(z_l) & \ldots & h_m(z_l)
\end{bmatrix} \]
and $J_{x,y}$ is a $k \times l$ matrix with entries in $\{-1, 0, 1\}$ and  block structure of the form
\[
J_{x,y} =  \left[ \begin{array}{c|c} * & -\Id_{r \times r} \\ \hline
* & * \end{array} \right],
\]
where $\Id_{r \times r}$ is the $r \times r$ identity matrix. It follows that $\rank J_{x,y} \geq r$. Moreover, since $z_1, \ldots, z_l$ are distinct and $h_1, \ldots, h_m$ is a $2k$-interpolating family, the matrix $V_{x,y}$ is of full rank, hence $\rank D_{x,y} = \rank J_{x,y}\geq r$, which ends the proof.
\end{proof}

\begin{proof}[Proof of Theorem~\rm\ref{thm:takens}] 
We proceed similarly as in the proof of Theorem~\ref{thm:embed}, using Lemma~\ref{lem: key_ineq_inter} instead of Lemma~\ref{lem: key_ineq_linear},   together with the suitable rank estimates coming from Lemma~\ref{lem:claim}.
In the same way as in the proof of Theorem~\ref{thm:embed}, we show that it is enough to check that the suitable set $X_\alpha$ exists for $\eta_m$-almost every $\alpha \in B_m(0,1)$. 

Applying Lemma~\ref{lem:dimh_fsigma} to the sets $\Per_p(T)$, $p = 1, \ldots,  k-1$ and (possibly zero) measures $\mu|_{\Per_p(T)}$, we find (possibly empty) disjoint $\sigma$-compact sets $X_1, \ldots, X_{k-1} \subset X$ such that
\[X_p \subset \Per_p(T),\quad \mu(\Per_p(T) \setminus X_p) = 0, \quad \mH^{\beta p}(X_p)=0 \quad \text{ for } p = 1, \ldots, k-1.\] Similarly, there exists a $\sigma$-compact set $X_k \subset X \setminus \bigcup \limits_{p=1}^{k-1}\Per_p(T)$ such that \[\mu\Big( \Big(X \setminus \bigcup \limits_{p=1}^{k-1}\Per_p(T)\Big) \setminus X_k \Big) = 0 \quad \text{ and } \quad \mH^{\beta k}(X_k)=0.\]
Note that $X_k$ contains both aperiodic and periodic points (with period at least $k$). Let 
\[
\tilde X = \bigcup \limits_{p=1}^k X_p.
\]
Then $\tilde X\subset X$ is a $\sigma$-compact set of full $\mu$-measure. Define
\[ A = \{ (x, \alpha) \in \tilde X \times B_m(0,1) :  \phi_\alpha^T (x) = \phi_\alpha^T (y) \text{ for some } y \in \tilde X \setminus \{ x \}\}. \]
The set $A$ is Borel by Lemma~\ref{lem:measurability}. For $x \in \tilde X$ and $\alpha \in B_m(0,1)$, denote, respectively, by $A_x$ and $A^\alpha$, the Borel sections
\[ A_x = \{ \alpha \in B_m(0,1) : (x, \alpha) \in A\},\quad A^\alpha = \{ x \in \tilde X : (x, \alpha) \in A \}. \]
Observe that to show the injectivity of $\phi_\alpha^T$ on a set of full $\mu$-measure, it is enough to prove $\eta_m(A_x) = 0$ for every $x \in \tilde X$, since then by Fubini's theorem (\cite[Thm. 8.8]{R87}), $(\eta_m \otimes \mu) (A) = 0$ and, consequently, $\mu(A^\alpha) = 0$ for $\eta_m$-almost every $\alpha \in B_m(0,1)$. As $\phi_\alpha^T$ is injective on $\tilde X \setminus A^\alpha$ and $\tilde X$ has full $\mu$-measure, the proof of the claim is finished. 

Fix $x \in \tilde X$. To show $\eta_m(A_x) = 0$, note that for $y \in \tilde X$,
\begin{equation}\label{e:matrix_form} \phi_\alpha^T(x) - \phi_\alpha^T(y) = D_{x,y}\alpha + w_{x,y}
\end{equation}
for
\[
w_{x,y} = \begin{bmatrix} h(x) - h(y) \\
h(Tx) - h(Ty)\\
\vdots \\
h(T^{k-1} x) - h(T^{k-1}y) \end{bmatrix}.
\]
Write $A_x$ as
\[ A_x = A_x^{\mathrm{orb}} \cup \bigcup \limits_{p=1}^k A_x^p,\]
where
\begin{align*}
A_x^{\mathrm{orb}} &= \{ \alpha \in B_m(0,1) : \phi_\alpha^T (x) = \phi_\alpha^T (y) \text{ for some } y \in \tilde X \cap \Orb(x) \setminus \{ x \}\},\\
A_x^p &= \{ \alpha \in B_m(0,1) : \phi_\alpha^T (x) = \phi_\alpha^T (y)  \text{ for some } y \in X_p \setminus \{x\} \}, \quad p = 1, \ldots, k.
\end{align*}
The set $A_x^{\mathrm{orb}}$ is Borel as a countable union of closed sets of the form
\begin{equation}\label{e:orb_sum}
\{ \alpha \in B_m(0,1) :  \phi_\alpha^T (x) = \phi_\alpha^T (y) \}, \quad y \in \tilde X \cap \Orb(x) \setminus \{ x \},
\end{equation}
while each set $A_x^p$ is Borel as a section of the set
\[
\{ (x, \alpha) \in \tilde X \times B_m(0,1) : \phi_\alpha^T (x) = \phi_\alpha^T (y)  \text{ for some } y \in X_p \setminus \{x\} \},
\]
which is Borel by Lemma~\ref{lem:measurability}.
To end the proof, it is enough to show that the sets $A_x^{\mathrm{orb}}$ and  $A_x^p$, $p = 1, \ldots, k$, have $\eta_m$ measure zero. 

To prove $\eta_m(A_x^{\mathrm{orb}}) = 0$ it suffices to check that the sets of the form \eqref{e:orb_sum} have $\eta_m$ measure zero. By \eqref{e:matrix_form}, we have
\[ \{ \alpha \in B_m(0,1) :  \phi_\alpha^T (x) = \phi_\alpha^T (y) \} = \{ \alpha \in B_m(0,1) : D_{x,y} \alpha = -w_{x,y} \} \]
and Lemma~\ref{lem:claim} gives $\rank D_{x,y}\geq 1$ whenever $y \neq x$, so each set of the form \eqref{e:orb_sum} is contained in an affine subspace of $\R^m$ of codimension at least $1$. Consequently, it has $\eta_m$ measure zero. 

Since $T$ is locally Lipschitz, $h, h_1, \ldots, h_m$ are locally $\beta$-H\"older and $X$ is separable, there exists a countable covering  $\mathcal V$ of $X$ by open sets in $\R^N$, such that for every $V \in \mathcal V$, the map $T$ is Lipschitz on $V$ and $h, h_1, \ldots, h_m$ are $\beta$-H\"older on $V$. Let $\mathcal U$ be the collection of all sets of the form $U = V_0 \cap T^{-1}(V_1) \cap \ldots \cap T^{-(k-1)}(V_{k-1})$, where $V_0, \ldots, V_{k-1} \in \mathcal V$. Then $\mathcal U$ is a countable covering of $X$ by open sets, and we can write $\mathcal U = \{U_j\}_{j\in \N}$. By definition, for every $j \in \N$ there exists $C_j > 0$ such that
\begin{align*}
\|T^{s+1}(y) - T^{s+1}(y')\| &\le C_j \|T^s(y)-T^s(y')\|,\\
\|h(T^s(y)) - h(T^s(y'))\| &\le C_j \|T^s(y)-T^s(y')\|^\beta, \\
\|h_r(T^s(y)) - h_r(T^s(y'))\| &\le C_j \|T^s(y)-T^s(y')\|^\beta
\end{align*}
for every $y, y' \in U_j \cap X$, $s \in \{0, \ldots, k-1\}$, $r \in \{1, \ldots, m\}$. By induction, it follows that
\begin{equation}\label{eq:lipschitz-holder}
\begin{aligned}
\|T^s(y) - T^s(y')\| &\le C_j^s \|y-y'\|,\\
\|h(T^s(y)) - h(T^s(y'))\| &\le C_j^{\beta s + 1}\|y-y'\|^\beta, \\
\|h_r(T^s(y)) - h_r(T^s(y'))\| &\le C_j^{\beta s + 1} \|y-y'\|^\beta
\end{aligned}
\end{equation}
for $y, y' \in U_j \cap X$, $s \in \{0, \ldots, k-1\}$, $r \in \{1, \ldots, m\}$. 

To prove $\eta_m(A_x^p) = 0$ for $p = 1, \ldots, k$, we follow the strategy used in \cite{SYC91} (see also \cite{Rob11}). 
Fix $n \in \N$ and for $j \in \N$ define
\begin{align*}
X_x^{p, n} &= \Big\{ y \in X_p : \sigma_p(D_{x,y}) \geq \frac{1}{n}\Big\},\\
A_x^{p, j, n} &= \{ \alpha \in B_m(0,1) :  \phi_\alpha^T (x) = \phi_\alpha^T (y) \text{ for some } y \in U_j \cap X_x^{p, n} \setminus \{x\}\},
\end{align*}
where $\sigma_p(D_{x,y})$ is the $p$-th largest singular value. Note that
singular values of given order depend continuously on the coefficients of the matrix, see e.g.~\cite[Corollary 8.6.2]{MatrixComputations}. Hence, the set $X_x^{p, n}$ is $\sigma$-compact as a closed subset of $X_p$ and by Lemma~\ref{lem:measurability}, the set $A_x^{p, j, n}$ is Borel.

By Lemma~\ref{lem:claim}, for every $y \in X_p \setminus \Orb(x)$ we have $\rank D_{x,y} \geq p$. This implies $\sigma_p(D_{x,y})>0$ (see e.g. \cite[Lemma 14.2]{Rob11}). Hence, 
\[ A_x^{p} \setminus A_x^{\mathrm{orb}} =  \bigcup \limits_{j=1}^{\infty }\bigcup \limits_{n=1}^{\infty } A_x^{p, j, n} \setminus A_x^{\mathrm{orb}}. \]
Consequently, it is enough to prove $\eta_m(A_x^{p, j, n} \setminus A_x^\mathrm{orb}) = 0$ for every $n \in \N$. 

Fix $\eps>0$. Since $\mH^{\beta p}(U_j \cap X_x^{p,n} \setminus \Orb(x)) \leq \mH^{\beta p}(X_p) = 0$, there exists a collection of balls $B_N(y_i, \eps_i)$, for $y_i \in U_j \cap X_x^{p,n}\setminus \Orb(x)$ and $0 < \eps_i < \eps$, $i \in \N$, such that
\begin{equation}\label{e:cover_takens} U_j \cap X_x^{p,n} \setminus \Orb(x) \subset \bigcup \limits_{i \in \N} B_N(y_i, \eps_i) \quad \text{ and } \quad \sum \limits_{i=1}^\infty \eps_i^{\beta p} \leq \eps.
\end{equation}
Take $\alpha \in A_x^{p, j, n} \setminus A_x^{\mathrm{orb}}$ and let $y \in U_j \cap X_x^{p, n} \setminus \Orb(x) $ be such that $\phi_\alpha^T(x) = \phi_\alpha^T(y)$. Then for $y_i$ with $y \in B(y_i, \eps_i)$ we have
\begin{equation}\label{e:D_xyj}
\begin{aligned}
	\|D_{x, y_i}\alpha + w_{x, y_i} \| &=  \|\phi_\alpha^T (x) - \phi_{\alpha}^T (y_i) \|
	= \|\phi_\alpha^T (y) - \phi_{\alpha}^T (y_i) \|\\
	&\leq \sqrt{\sum_{s=0}^{k-1} \Big(\|h(T^sy) - h(T^sy_i)\| + \sum_{r=1}^m \alpha_r \|h_r(T^sy) - h_r(T^sy_i)\|\Big)^2}\\
	&\leq M_j \|y-y'\|^\beta \le M_j \eps_i^\beta
\end{aligned}
\end{equation}
for
\[
M_j = (1 + \sqrt{m})\;\sqrt{\sum_{s=0}^{k-1} C_j^{2(\beta s + 1)}}, 
\]
by \eqref{eq:lipschitz-holder} and the fact $\alpha \in B_m(0,1)$. By \eqref{e:D_xyj}, 
\[ A_x^{p, j, n} \setminus A_x^{\mathrm{orb}} \subset \bigcup \limits_{i \in \N} \{  \alpha \in B_m(0,1) : \|D_{x, y_i}\alpha + w_{x, y_i} \| \leq M_j\eps_i^\beta \}. \]
Since for every $i \in \N$ we have $\sigma_p(D_{x, y_i}) \geq 1/n$, we can apply Lemma~\ref{lem: key_ineq_inter} and \eqref{e:cover_takens} to obtain
\[ \eta_m(A_x^{p, j, n} \setminus A_x^{\mathrm{orb}}) \leq \sum \limits_{i =1 }^\infty C_{m,k} \frac{M_j^p\eps_i^{\beta p}}{1 / n^p} \leq  C_{m,k}M_j^pn^p\eps. \]
Since $\eps>0$ was arbitrary, we conclude that $\eta_m(A_x^{p, j, n} \setminus A_x^{\mathrm{orb}}) = 0$, so in fact $\eta_m(A_x^{p, j, n}) = 0$. This ends the proof of Theorem~\ref{thm:takens}. 
\end{proof}

\begin{proof}[Proof of Remark~\rm\ref{rem:invariant}]
Suppose that the measure $\mu$ is $T$-invariant. Then it is easy to check that the set 
\[
\tilde X_\alpha = \bigcap \limits_{n = 0}^{\infty} T^{-n} (X_\alpha) 
\]
is a Borel subset of $X_\alpha$ of full $\mu$-measure satisfying $T(\tilde X_\alpha) \subset \tilde X_\alpha$. Hence, to show (a), it suffices to replace the set $X_\alpha$ by $\tilde X_\alpha$. 

In the case when $\mu$ is additionally finite, we first remark that the measure $\mu$ is also forward invariant, i.e.~$\mu(T(Y)) = \mu(Y)$ for Borel sets $Y \subset X$. Note that if $Y$ is Borel, then so is $T(Y)$ as the image of a Borel set under a continuous and injective mapping (see e.g. \cite[Theorem~15.1]{K95}). Using this together with the invariance of $\mu$ and the injectivity of $T$, we check that 
\[ 
\tilde X_\alpha = \bigcap \limits_{n \in \Z} T^{-n} (X_\alpha) 
\]
is a Borel subset of $X_\alpha$ of full $\mu$-measure satisfying $T(\tilde X_\alpha) = \tilde X_\alpha$. This gives (b). Notice that the  finiteness of $\mu$ is indeed necessary, as for $X = \N,$ $T(x) = x+1$ and $\mu$ the counting measure, there does not exist a set $Y \subset X$ of full $\mu$-measure satisfying $T(Y)=Y$.

To show (c), suppose that $\mu$ is $T$-invariant and ergodic. 
Obviously, we can assume that the $\mu$-measure of the set of all periodic points of $T$ is positive (including $+\infty$). Then there exists $p \in \N$ such that the measure of the set $P$ of all $p$-periodic points of $T$ is positive (including $+\infty$).

Suppose first that $\mu$ restricted to $P$ is non-atomic. Then there exists a Borel set $Y \subset P$ with $0 < \mu(Y) < \mu(X)/p$. Let $Z = Y \cup T^{-1}(Y) \cup \ldots \cup T^{-(p-1)}(Y)$. Then $0 < \mu(Z) < \mu(X)$ and, by the injectivity of $T$, we have $T^{-1}(Z) = Z$, which contradicts the ergodicity of $\mu$.  

Suppose now that $\mu$ has an atom in $P$. Since $\mu$ is a Borel $\sigma$-finite measure in a Euclidean space, this is equivalent to the fact that $\mu(\{x\}) > 0$ for some $x \in P$. Let $\mathcal O$ be the periodic orbit of $x$. Again by the injectivity of $T$, we have $T^{-1}(\mathcal O) = \mathcal O$, so by the ergodicity of $\mu$, the set $\mathcal O$ has full $\mu$-measure. This means that $\mu$ is supported on a set of Hausdorff dimension $0$, which obviously gives (c).
\end{proof}

The original Takens delay embedding theorem states that for given finite dimensional $C^2$ manifold $M$ and generic pair of $C^2$-diffeomorphism $T\colon M \to M$ and $C^2$-function $h\colon  M \to \R$, the corresponding delay-coordinate map $\phi\colon  M \to \R^k$, $\phi(x) = (h(x), h(Tx), \ldots, h(T^{k-1}x))$ is a $C^2$-embedding (an injective immersion) as long as $k > 2\dim M$. It was followed by the box-counting dimension version of Sauer, Yorke and Casdagli (Theorem~\ref{thm:standard_takens}) and subsequently by the infinite-dimensional result of \cite{Rob05} (see also  \cite[Section 14.3]{Rob11}). Refer to \cite{NV18} for a version of Takens' theorem with a fixed observable and perturbation performed on the dynamics. Takens' theorem involving Lebesgue covering dimension on compact metric spaces and a continuous observable was proved in \cite{Gut16} (see \cite{GQS18} for a detailed proof). See also \cite{1999delay, CaballeroEmbed} for Takens theorem for deterministically driven smooth systems and \cite{StarkEmbedSurvey,  StarkStochEmbed} for stochastically driven smooth systems. 

\begin{examplex}\label{ex:no_linear_takens} It turns out that linear perturbations are not sufficient for Theorems~\ref{thm:standard_takens} and~\ref{thm:takens}, i.e.~it may happen that $\phi_L = (\phi(x) + Lx, \ldots, \phi(T^{k-1}x) + LT^{k-1}x)$ is not (almost surely) injective for a generic linear map $L \colon  \R^N \to \R$. As an example, let $X = B_2(0,1)$, fix $a \in (0,1)$ and define $T\colon X \to X$ as
	\[ T(x) = ax. \]
	Then $T$ is a Lipschitz injective transformation on the unit disc $X \subset \R^2$ with zero being the unique periodic point. Fix $\phi \equiv 0$. We claim that there is no linear observable $L \colon  \R^2 \to \R$ which makes the delay map injective, i.e. for every $k \in \N$ and every $v \in \R^2$ the transformation $x \mapsto \phi_v^T(x)= (\langle x, v \rangle, \langle Tx, v \rangle, \ldots, \langle T^{k-1} x, v \rangle) \in \R^k$ is not injective on $X$. This follows from the fact that for each $1$-dimensional linear subspace $W \subset \R^2$ the set $W \cap X$ is $T$-invariant, hence $\phi_v^T = 0$ on an infinite set $\Ker(\langle \cdot, v \rangle) \cap X$. We have seen that $\phi_v^T$ is not injective for any $v \in \R^2$. No we will see that it also not almost surely injecitve for $\mu$ being the Lebesgue measure on $X$. Note that for $v \in \R^2$ and $c \in \R$, the segment $W_c = \{ z \in X : \langle z, v \rangle = c  \}$ satisfies $T(W_c) \subset W_{ac} $, hence all points on $W_c$ will have the same observation vector $(\langle x, v \rangle, \langle Tx, v \rangle, \ldots, \langle T^{k-1} x, v \rangle) = (c, ac, a^2c, \ldots, a^{k-1}c)$. Therefore, a set $X_v \subset X$ on which $\phi_v^T$ is injective can only have one point on each of the parallel segments $W_c$ contained in $X$. However, such a set $X_v$ cannot be of full Lebesgue measure. Note that the above example can be easily modified to make $T$ a homeomorphism.
\end{examplex}

\section{Examples} \label{sec:examples}
In this section we present two examples which illustrate the usage of Theorem~\ref{thm:embed}. Let us begin with fixing some notation. For $x \in [0,2)$ we will write
\[
x = x_0.x_1 x_2\ldots,
\]
where $x_0.x_1 x_2 \ldots$ is the \emph{binary expansion} of $x$, i.e.
\[
x = \sum \limits_{j=0}^{\infty} \frac{x_j}{2^j}, \quad x_0, x_1, x_2,\ldots \in \{0,1\}.
\]
For a dyadic rational we agree to choose its eventually terminating expansion, i.e.~the one with $x_j = 0$ for $j$ large enough. Let $\pi \colon  \{0,1\}^\N \to [0,1]$ be the coding map
\[ \pi(x_1, x_2, \ldots) = \sum \limits_{j=1}^{\infty} \frac{x_j}{2^j}. \]

\subsection{A modified Kan example}
In the Appendix to \cite{SYC91}, Kan presented an example of a compact set $K \subset \R^N$ with $\hdim K =0$ and such that every linear transformation $L \colon  \R^N \to \R^{N-1}$ fails to be injective on $K$ (see also Remark~\ref{rem:hdim_does_not_work}). It follows from Theorem~\ref{thm:embed}, that whenever we are given a Borel $\sigma$-finite measure $\mu$ on such a set, then almost every linear transformation $L \colon  \R^N \to \R$ is injective on a set of full $\mu$-measure. To illustrate this, we construct a $\sigma$-compact set $X \subset \R^2$ with $\hdim X =0$, which is a slight modification of Kan's example, equipped with a natural Borel $\sigma$-finite measure $\mu$, such that no linear transformation $L \colon  \R^2 \to \R$ is injective on $X$, while for almost every $L$ we explicitly show a set $X_L \subset X$ of full $\mu$-measure, such that $L$ is injective on $X_L$.

Following \cite[Appendix]{SYC91}, we begin with constructing compact sets $A, B \subset [0,1]$ such that
\begin{equation}\label{eq:AB_dim1}
\hdim A = \ldim A = \hdim B =\ldim B =0\quad (\text{hence } \hdim(A \cup B) = 0),
\end{equation} and
\begin{equation}\label{eq:AB_dim2}
\udim A  = \udim B = 1,  \quad \ldim(A \cup B) = \udim(A \cup B)= 1.
\end{equation}
To this aim, let $M_k$, $k \ge 0$, be an increasing sequence of positive integers such that $M_0 = 1$ and $M_k \nearrow \infty$ with $\lim \limits_{k \to \infty} \frac{M_{k+1}}{M_k} = \infty$. Define
\begin{alignat*}{2}
\widetilde A &= \big\{ (x_1,x_2,\ldots) \in \{0,1\}^\N : &\text{ for every even } k, \; &x_j = 0 \text{ for all } j \in [M_{k}, M_{k+1})\\
& & &\text{or } x_j = 1 \text{ for all } j \in [M_{k}, M_{k+1}) \big\},\\
\widetilde B &= \big\{ (x_1,x_2,\ldots) \in \{0,1\}^\N : &\text{ for every odd } k, \;  &x_j = 0 \text{ for all } j \in [M_{k}, M_{k+1})\\
& & &\text{or } x_j = 1 \text{ for all } j \in [M_{k}, M_{k+1}) \big\},
\end{alignat*}
and set
\[ A = \pi(\widetilde A), \quad B = \pi(\widetilde B).\]
It is a straightforward calculation to check that $A$ and $B$ satisfy \eqref{eq:AB_dim1} and \eqref{eq:AB_dim2} (see \cite[Appendix]{SYC91}, \cite[Example 7.8]{falconer2014fractal} or \cite[Section 6.1]{Rob11}). Define $X \subset \R^2$ as
\[ X =  \Big(\{0\} \times \bigcup \limits_{n \in \Z}(A+n)\Big) \cup   \Big(\{1\} \times \bigcup \limits_{n \in \Z}(B+n)\Big). \]
By \eqref{eq:AB_dim1}, we have $\hdim X = 0$. The following two propositions describe the embedding properties of the set $X$.

\begin{prop}\label{prop:X_no_injection}
	No linear transformation $L\colon \R^2 \to \R$ is injective on $X$.
\end{prop}
\begin{proof} The map $L$ has the form $L(x,y) = \alpha x + \beta y$ for $\alpha, \beta \in \R$. Obviously, we can assume $\beta \neq 0$. Note that the points 
\[
u = (0,a+n),\quad v = (1, b+m), \qquad \text{for }a\in A,\ b \in B,\ n,m \in \Z
\]
are in $X$ and
	\begin{equation}\label{eq:L_kernel} L(u) = L(v) \quad\text{ if and only if }\quad b-a = z,\end{equation}
where
\[
z = -\frac{\alpha}{\beta} + n - m.
\]
For given $\alpha$ and $\beta$, choose $n,m \in \Z$ such that $z \in [0,1)$.
Consider the binary expansion $z = 0.z_1z_2\ldots$ and define
\[
a=0.a_1a_2\ldots \in A,\quad b = 0.b_1b_2\ldots \in B
\]
setting
	\begin{equation}\label{eq:AB_decomposition}
	\begin{aligned}
	&a_j = 0, \quad &&b_j = z_j \quad &&\text{for } j \in [M_{k}, M_{k+1}), &&\text{if } k \text{ is even}, \\
	&a_j = 1 - z_j, \quad &&b_j = 1 \quad &&\text{for } j \in [M_{k}, M_{k+1}), &&\text{if }k \text{ is odd}
	\end{aligned}
	\end{equation}
	(if all $b_j$ are equal to $1$, we set $b=1$). Then $z = b-a$ and (\ref{eq:L_kernel}) implies that $L$ is not injective on $X$.
\end{proof}
Let us now define a natural Borel $\sigma$-finite measure $\mu$ on $X$, starting from a pair of probability measures $\nu_1, \nu_2$ on $\widetilde A$ and $\widetilde B$, respectively. Let
\[ \nu_1 = \bigotimes \limits_{k=0}^{\infty} \textbf{p}_k, \quad  \nu_2 = \bigotimes \limits_{k=0}^{\infty} \textbf{q}_k, \]
where $\textbf{p}_k$ and $\textbf{q}_k$ are probability measures on $\{0,1\}^{M_{k+1}-M_k}$ given as
\[ \textbf{p}_k = \begin{cases} \frac{1}{2} \delta_{(0, \ldots, 0)} + \frac{1}{2} \delta_{(1, \ldots, 1)} &\text{if } k \text{ is even}\\
\big(\frac{1}{2} \delta_{0} + \frac{1}{2} \delta_{1}\big)^{\otimes(M_{k+1}-M_k)} &\text{if } k \text{ is odd}
\end{cases}, \quad \textbf{q}_k = \begin{cases} \big(\frac{1}{2} \delta_{0} + \frac{1}{2} \delta_{1}\big)^{\otimes(M_{k+1}-M_k)} &\text{if } k \text{ is even} \\
\frac{1}{2} \delta_{(0, \ldots, 0)} + \frac{1}{2} \delta_{(1, \ldots, 1)} &\text{if } k \text{ is odd}\\
\end{cases} \]
and the symbol $\delta_a$ denotes the Dirac measure at $a$.
Then $\supp \nu_1 = \widetilde A,\ \supp \nu_2 = \widetilde B$, hence defining
\[ \mu_1 = \pi_* (\nu_1), \quad \mu_2 = \pi_* (\nu_2), \]
we obtain probability measures on $A, B$, respectively, with $\supp \mu_1 = A$, $\supp \mu_2 = B$. Finally, let
\[ \mu = \sum \limits_{n \in \Z} \delta_{0} \otimes (\tau_n)_* \mu_1 + \sum \limits_{n \in \Z} \delta_{1} \otimes (\tau_n)_* \mu_2, \]
where $\tau_n \colon  \R \to \R,\ \tau_n(x) = x +n,\ n \in \Z$. Clearly, $\mu$ is a Borel  $\sigma$-finite measure with $\supp \mu=X$.

For $a \in A,\ b \in B$ let
	\begin{alignat*}{3}
	A_a &= \big\{ x \in A \setminus \{ 1 \} : \; &&x + a = z_0.z_1z_2\ldots \text{ such that the sequence } (z_0, z_1, \ldots)\\
	&&&\text{is constant on } [M_{k}, M_{k+1}-1) \cap \N \text{ for every odd } k \big\},\\
	B_b &= \big\{ x \in B \setminus \{1\} : \; &&x + b = z_0.z_1z_2\ldots \text{ such that the sequence } (z_0, z_1, \ldots)\\
	&&&\text{is constant on } [M_{k}, M_{k+1}-1) \cap \N \text{ for every even } k \big\}.
	\end{alignat*}

\begin{lem}\label{lem:binary_nonconstant}
For every $a \in A,\ b \in B$, we have $\mu_1(A_a) = \mu_2(B_b) = 0$.
\end{lem}

\begin{proof}
	Fix $b=b_0.b_1b_2\ldots \in B$. We will show $\mu_2(B_b) = 0$ (the fact $\mu_1(A_a)=0$ can be proved analogously). The proof proceeds by showing that for each even $k$, the vector $(x_{M_{k}}, \ldots, x_{M_{k+1}-2})$, where $x = x_0.x_1x_2\ldots \in B_b$, can assume at most four values. This will imply $\mu_2(B_b) \leq 8 \cdot 2^{-(M_{k+1} - M_{k})}$ for each even $k$ and, consequently, $\mu_2(B_b) = 0$. To show the assertion, fix an even $k$ and let 
	\[\xi = \sum \limits_{j=M_{k+1}-1}^\infty \frac{x_j+b_j}{2^j}.\]
	Note that $\xi < 2^{-(M_{k+1}-3)}$ (as $\xi<2$ and we exclude expansions with digits eventually equal to $1$). Hence, $\xi = \xi_0.\xi_1 \xi_2\ldots$ with $\xi_j = 0 $ for $j \leq M_{k+1}-3$. Note that, since $b$ is fixed, the values of $\xi_{M_{k+1}-2} \in \{0,1\}$ and $(x_{M_{k}}+b_{M_{k}}, \ldots, x_{M_{k+1}-2}+b_{M_{k+1}-2}) \in \{ (0,\ldots,0),(1,\ldots,1) \}$ determine uniquely the value of $(x_{M_{k}}, \ldots, x_{M_{k+1}-2})$. Therefore, $(x_{M_{k}}, \ldots, x_{M_{k+1}-2})$ can assume at most four values.
\end{proof}

Now for Lebesgue almost every linear transformation $L\colon \R^2 \to \R$ we will construct a set $X_L \subset X$ of full $\mu$-measure, such that $L$ is injective on $X_L$. As previously, write $L(x,y) = \alpha x + \beta y$ for $\alpha,\beta \in \R$. Neglecting a set of zero Lebesgue measure, we can assume $\beta \neq 0$. Let $l \in \Z$ be such that 
\begin{equation}\label{eq:z}
z = -\frac{\alpha}{\beta} + l \text{ belongs to } [0,1).
\end{equation}
Similarly as in \eqref{eq:AB_decomposition}, we can write
	\begin{equation}\label{eq:A'B'_decomposition}
	z = a' - b', \quad  z - 1 = a'' - b'' \quad\text{for some } a',a'' \in A, \;  b',b'' \in B.
	\end{equation}
	Let
	\[ X_L = \Big( \{0\} \times \bigcup \limits_{n \in \Z}(A+n)\Big) \cup  \Big(\{1\} \times \bigcup \limits_{n \in \Z}\big((B \setminus (B_{b'} \cup B_{b''} \cup \{ 1 \}))+n\big)\Big). \]
	Then $X_L \subset X$ and Lemma~\ref{lem:binary_nonconstant} implies that $X_L$ has full $\mu$-measure.

\begin{prop}\label{prop:X_almost_sure_injection} For every $\alpha \in \R,\ \beta \in \R \setminus \{0\}$, the linear transformation $L\colon \R^2 \to \R$, $L(x,y) = \alpha x + \beta y$, is injective on $X_L$.
\end{prop}

For the proof of the proposition we will need the following simple lemma. The proof is left to the reader.

\begin{lem}\label{lem:binary_piecewise_constant}
	Let $x=x_0.x_1x_2\ldots\in [0,1]$, $y = y_0.y_1y_2\ldots \in [0,1]$, $M, N \in \N$, $M<N-1$, be such that $x+y < 2$ and sequences $(x_M, \ldots, x_N)$ and $(y_M, \ldots, y_N)$ are constant. Then $x + y = z_0.z_1z_2\ldots$, where the sequence $(z_M, \ldots, z_{N-1})$ is constant.
\end{lem}

\begin{proof}[Proof of Proposition~\rm\ref{prop:X_almost_sure_injection}]  Assume, on the contrary, that there exist points $u, v \in X_L$ such that 
$L(u) = L(v)$. As $\beta \neq 0$, we cannot have $u, v \in \{0\} \times \R$ or $u, v \in \{1\} \times \R$. Hence, we can assume $u \in \{0\} \times \R$, $v \in \{1\} \times \R$. Then, following the previous notation, we have $u = (0,a+n)$, $v = (1, b+m)$ for $a\in A$, $b \in B \setminus (B_{b'} \cup B_{b''} \cup \{1\})$, $n,m \in \Z$. Note that $b-a \in [-1,1)$, so by \eqref{eq:L_kernel}, we have
	\[ b - a = z \quad \text{ or } \quad  b - a = z - 1, \]
	for $z$ from \eqref{eq:z}, and \eqref{eq:A'B'_decomposition} implies
	\[ b - a = a' - b' \quad \text{ or } \quad b - a = a'' - b''.  \]
	Hence,
	\[ a + a' = b + b' \quad \text{ or }\quad  a + a'' = b + b''.  \]
	This is a contradiction, as Lemma~\ref{lem:binary_piecewise_constant} implies that the binary expansion sequences of $a + a'$ and $a + a''$ are constant on $[M_{k}, M_{k+1}-1) \cap \N$ for every even $k$, while by the condition $b \in B \setminus (B_{b'} \cup B_{b''} \cup \{1\})$, the binary expansion sequences of $b + b'$ and $b + b''$ are not constant on $[M_{k}, M_{k+1}-1) \cap \N$ for some even $k$. 
\end{proof}

\subsection{\boldmath Measure with $\hdim \mu < \underline{\dim}_{\,\it MB\,} \mu$}

To show that Theorem~\ref{thm:embed} is an actual strengthening of Theorem~\ref{thm:riegler}, we present an example of a measure $\mu$, for which 
$\hdim \mu  < \lmodbdim \mu$. More precisely, we show the following.

\begin{thm}\label{thm:hdim<lmodbdim}
	There exists a Borel probability measure $\mu$ on $[0,1]^2$, such that $\hdim \mu = 1$ and $\lmodbdim \mu=2$.
\end{thm}
To begin the construction of $\mu$, fix an increasing sequence of positive integers $N_k$, $k \in \N$, such that $N_k \nearrow \infty$ with $\frac{S_k}{S_{k+1}} \leq \frac{1}{k+1}$, where $S_k = \sum \limits_{j=1}^k N_j$. Consider the probability distributions $\textbf{p}_0, \textbf{p}_1$ on $\{0,1\}$ given by 
\[
\textbf{p}_0 (\{0\}) = 0,\ \textbf{p}_0 (\{1\}) = 1, \qquad \textbf{p}_1 (\{0\}) = \textbf{p}_1 (\{1\}) =  \frac{1}{2}.
\]
For $y = 0.y_1 y_2 \ldots \in[0,1]$ (in this subsection we  assume that the binary expansion of $1$ is $0.111\ldots$), define the probability measure $\nu_y$ on $\{0,1\}^\N$ as the infinite product
\[ \nu_y = \bigotimes \limits_{j=1}^ \infty \bigotimes \limits_{i=1}^{N_j}\textbf{p}_{y_j}.\]
Further, let $\mu_y$ be the Borel probability measure on $[0,1]$ given by
\[ \mu_y = \pi_* \nu_y.\]
Finally, let $\mu$ be the Borel probability measure on $[0,1]^2$ defined as
\[ \mu = \int \limits_{[0,1]} \mu_y d\Leb(y), \quad\text{ i.e. } \mu(A) = \int \limits_{[0,1]} \mu_y(A^y) d\Leb(y) \quad\text{for a Borel set } A \subset [0,1]^2, \]
where $A^y = \{ x \in [0,1] : (x,y) \in A \}$. It is easy to see that $\mu$ is well-defined, as the function $y \mapsto \mu_y(A^y)$ is measurable for every Borel set $A \subset [0,1]^2$.

The proof of Theorem \ref{thm:hdim<lmodbdim} is based on the analysis of the local dimension of $\mu$, defined in terms of dyadic squares (rather then balls). For $n \in \N$ and $x_1, \ldots, x_n \in \{ 0,1 \}$ let $[x_1, \ldots, x_n]$ denote the dyadic interval corresponding to the sequence $(x_1, \ldots ,x_n)$, i.e.
\[[x_1, \ldots, x_n] =
\begin{cases}
\Big[\sum \limits_{j=1}^n \frac{x_j}{2^j}, \sum \limits_{j=1}^n \frac{x_j}{2^j} + \frac{1}{2^n}\Big) &\text{if } \sum \limits_{j=1}^n \frac{x_j}{2^j} + \frac{1}{2^n} < 1\\
\big[1 - \frac{1}{2^n}, 1\big] &\text{otherwise.}
\end{cases}
\]
Under this notation, for $n \in \N$ and $(x,y) \in [0,1]^2$ let $D_n(x,y)$ be the dyadic square of sidelength $2^{-n}$ containing $(x,y)$, i.e.
\[ D_n(x,y) = [x_1, \ldots, x_n] \times [y_1, \ldots, y_n], \quad\text{where } x = 0.x_1 x_2 \ldots \text{ and }y=0.y_1 y_2 \ldots .  \]
Recall that the box-dimensions can be defined equivalently in terms of dyadic squares. Precisely, let
$N'(X, 2^{-n})$ be the number of dyadic squares $D_n(x,y)$ of sidelength $2^{-n}$ intersecting $X$. Then (see e.g.~\cite[Section 2.1]{falconer2014fractal})
\begin{equation}\label{eq:bdim_dyadic} \ldim(X) = \liminf \limits_{n \to \infty} \frac{\log N'(X, 2^{-n})}{n \log 2} \text{ and }  \udim(X) = \limsup \limits_{n \to \infty} \frac{\log N'(X, 2^{-n})}{n \log 2}.\end{equation}
For a Borel finite measure $\mu$ on $[0,1]^2$ and $(x,y) \in [0,1]^2$ define the \emph{lower} and \emph{upper local dimension} of $\mu$ at $(x,y)$ as
\[ \underline{d}(\mu,(x,y)) = \liminf \limits_{n \to \infty} \frac{- \log \mu(D_n(x,y))}{n \log 2}, \quad \overline{d}(\mu,(x,y)) = \limsup \limits_{n \to \infty} \frac{- \log \mu(D_n(x,y))}{n \log 2}. \]
It is well-known (see e.g.~\cite[Propositions 3.10 and 3.20]{HochmanNotes}) that
\begin{equation}\label{eq:hdim_local}
\hdim \mu = \underset{(x,y) \sim \mu}{\mathrm{ess\ sup}}\ \underline{d}(\mu, (x,y)).
\end{equation}
The following lemma gives estimates on the measure of dyadic squares at suitable scales.
\begin{lem}\label{ref:lem_square_prob_ineqalities}
	Let $x=0.x_1x_2\ldots,\in [0,1]$, $y=0.y_1y_2\ldots \in [0,1]$, $n \in \N$ and $D = D_n(x,y) = [x_1, \ldots, x_n] \times [y_1, \ldots, y_n]$. Let $k \in \N$ be such that $S_k < n \leq S_{k+1}$. Then the following hold.
	\begin{enumerate}[\rm (a)]
		\item If $y_{k} = y_{k+1}=1$, then $\mu(D) \leq 2^{-(2 - \frac{1}{k})n}$.
		\item If $n = S_{k+1}$ and $y_{k+1} = 0$, then either $\mu(D) = 0$ or $\mu(D) \geq 2^{-(1 + \frac{1}{k+1})n}$.
	\end{enumerate}
\end{lem}
\begin{proof}
	Note that for $y' = 0.y'_1 y'_2\ldots \in [0,1]$ such that $(y'_1, \ldots, y'_n) = (y_1, \ldots, y_n)$ we have
	\begin{equation}\label{eq:square_prob}
	\begin{aligned}
	\mu_{y'}(D^{y'}) = \mu_{y'}([x_1, \ldots, x_n]) = \ &\textbf{p}_{y'_1}(\{x_1\}) \cdots \textbf{p}_{y'_1}(\{x_{S_1}\}) \textbf{p}_{y'_2}(\{x_{S_1 + 1}\}) \cdots \textbf{p}_{y'_2}(\{x_{S_2}\})\\
	& \cdots \textbf{p}_{y'_{k+1}}(\{x_{S_k + 1}\}) \cdots \textbf{p}_{y'_{k+1}}(\{x_{n}\}).
	\end{aligned}
	\end{equation}
	Moreover, as $k < n$, the value of $\mu_{y'}(D^{y'})$ depends only on $(y_1, \ldots, y_n)$ and $(x_1, \ldots, x_n)$. Using (\ref{eq:square_prob}), we can prove both assertions of the lemma, as follows.
	\paragraph*{Ad (a)}
	
	\
	
		If $y_{k} = y_{k+1}=1$, then for $j \in \{S_{k-1} + 1, \ldots, n\}$ we have $\textbf{p}_{y_l}(x_j) = \frac{1}{2}$, where $l \in \{ k, k+1\}$ is such that $S_{l-1}<j \leq S_{l}$. Therefore, in the product \eqref{eq:square_prob} there is at least $n - S_{k-1}$ terms equal to $\frac{1}{2}$. Consequently,
		\[ \mu_{y'}(D^{y'}) \leq 2^{-(n - S_{k-1})} = 2^{-(1 - \frac{S_{k-1}}{n})n} \leq 2^{-(1 - \frac{S_{k-1}}{S_k})n} \leq 2^{-(1 - \frac{1}{k})n},\]
		hence
		\[ \mu(D) = \int \limits_{[y_1, \ldots, y_n]} \mu_{y'}(D^{y'})d\Leb(y') \leq \Leb([y_1, \ldots, y_n]) 2^{-n(1 - \frac{1}{k})} = 2^{-n(2 - \frac{1}{k})}.  \]

		\paragraph*{Ad (b).}
		
		\
		
		Assume that $\mu(D) \neq 0$. Then all the terms in (\ref{eq:square_prob}) have to be non-zero, so every term is equal to either $\frac{1}{2}$ or $1$. Moreover, as $y_{k+1} = 0$ and $n = S_{k+1}$, we have
		\[ \textbf{p}_{y_{k+1}}(\{x_{S_k + 1}\}) \cdots \textbf{p}_{y_{k+1}}(\{x_{n}\}) = 1\]
		and, consequently,
		\begin{align*}
		\begin{split}
		\mu(D)  = \ &2^{-n} \textbf{p}_{y_1}(\{x_1\}) \cdots \textbf{p}_{y_1}(\{x_{S_1}\}) \textbf{p}_{y_2}(\{x_{S_1 + 1}\}) \cdots \textbf{p}_{y_2}(\{x_{S_2}\})\\
		& \cdots \textbf{p}_{y_{k}}(\{x_{S_{k-1} + 1}\}) \cdots \textbf{p}_{y_{k}}(\{x_{S_k}\}) \geq 2^{-n-S_k} = 2^{-(1 + \frac{S_k}{S_{k+1}})n} \geq 2^{-(1 + \frac{1}{k+1})n}.
		\end{split}
		\end{align*}
	
\end{proof}
Now we are ready to give the proof of Theorem~\ref{thm:hdim<lmodbdim}.
\begin{proof}[Proof of Theorem~\rm\ref{thm:hdim<lmodbdim}] We begin by proving $\hdim \mu =1$. Note that $\hdim \mu \geq 1$, as $\mu$ projects under $[0,1]^2 \ni (x,y) \mapsto y \in [0,1]$ to the Lebesgue measure, so it is sufficient to show $\hdim \mu \leq 1$. By (\ref{eq:hdim_local}), it is enough to prove that $\underline{d}(\mu, (x,y)) \leq 1$ for $\mu$-almost every $(x,y) \in [0,1]$. Note that for Lebesgue almost every $y = 0.y_1 y_2\ldots\in [0,1]$, the sequence $(y_1, y_2, \ldots)$ contains infinitely many zeros. Hence, it is sufficient to show $\underline{d}(\mu, (x,y)) \leq 1$ for $\mu_y$-almost every $x\in [0,1]$, assuming that $y\in [0,1]$ has this property.  Moreover, for $\mu_y$-almost every $x \in [0,1]$, we have $\mu(D_n(x,y)) > 0$ for all $n \in \N$ (see \eqref{eq:square_prob}). For such $x$, by Lemma~\ref{ref:lem_square_prob_ineqalities}(b), we have
	\[ \underline{d}(\mu,(x,y)) \leq \liminf \limits_{k \to \infty} \frac{-\log \mu(D_{S_{n_k}}(x,y))}{S_{n_k} \log 2} \leq \lim \limits_{k \to \infty}\frac{(1 + \frac{1}{n_k})S_{n_k}}{S_{n_k}} = 1. \]
	Therefore, $\hdim \mu \leq 1$, so in fact $\hdim \mu = 1$.

	Let us prove now $\lmodbdim \mu = 2$. Since $\mu$ is supported on $[0,1]^2$, it suffices to show $\lmodbdim \mu \geq 2$. Let $A \subset [0,1]^2$ be a Borel set with $\mu(A)>0$. We show $\ldim A \geq 2$. Note that there exists $c > 0$ such that the set
	\begin{equation}\label{eq:B_measure}
	  B = \{ y \in [0,1] : \mu_y(A^y) \geq c \}
	\end{equation}
satisfies $\Leb(B) > 0.$ Fix $\eps \in (0, \frac{1}{4})$. By the Lebesgue density theorem (see e.g.~\cite[Corollary 3.16]{HochmanNotes}), there exists a dyadic interval $I \subset [0,1]$ such that
	\begin{equation}\label{eq:I_density}
	\frac{\Leb(B \cap I)}{|I|} \geq 1 - \eps,
	\end{equation}
	where $|I| = 2^{-N}$ is the length of $I$. Fix $k \geq N+2$ and $n \in \{S_k + 1,  \ldots, S_{k+1} \}$. Consider the collection $\mathcal{C}_n$ of dyadic intervals of length $2^{-n}$ defined as
	\[ \mathcal{C}_n = \{ [y_1, \ldots, y_n] : y_k = y_{k+1} = 1 \text{ and } [y_1, \ldots, y_n] \cap B \cap I \neq \emptyset \}.\]
	By \eqref{eq:I_density}, we have
	\begin{equation}\label{eq:BC_n_measure} \Leb\Big(B \cap \bigcup \mathcal{C}_n\Big) \geq \Big(\frac{1}{4} - \eps\Big)2^{-N}.
	\end{equation}
	Let 
	\[
	A_n = A \cap \Big([0,1] \times \Big(B \cap \bigcup \mathcal{C}_n\Big)\Big).
	\]
	Then $A_n \subset A$ and (\ref{eq:B_measure}) together with (\ref{eq:BC_n_measure}) imply
	\begin{equation}\label{eq:A0_measure}
	\mu(A_n) = \int \limits_{B \cap \bigcup \mathcal{C}_n} \mu_y(A^y) d\Leb(y) \geq c \Big(\frac{1}{4} - \eps\Big)2^{-N}.
	\end{equation}
	Note that the above lower bound does not depend on $k$ and $n$.
	Let $N'(A_n, 2^{-n})$ be the number of dyadic squares of sidelength $2^{-n}$ intersecting $A_n$. If $D = I_1 \times I_2$ is a dyadic square of sidelength $2^{-n}$ intersecting $A_n$, then $I_2 \in\mathcal{C}_n$, hence by Lemma~\ref{ref:lem_square_prob_ineqalities}(a) we have
	\[ \mu(D) \leq 2^{-(2 - \frac{1}{k})n}. \]
	As any two dyadic squares of the same sidelength are either equal or disjoint, \eqref{eq:A0_measure} gives
	\[ N'(A, 2^{-n}) \geq N'(A_n, 2^{-n}) \geq c \Big(\frac{1}{4} - \eps\Big)2^{-N  + (2 - \frac{1}{k})n}. \]
	Since $k$ and $n$ can be taken arbitrary large, invoking (\ref{eq:bdim_dyadic}) gives $\ldim A \geq 2$. Hence, $\lmodbdim \mu \geq 2$, so in fact $\lmodbdim \mu = 2$.
\end{proof}
\begin{rem}
	Note that as
	\[ \underset{z \sim \mu}{\mathrm{ess\ sup}}\ \underline{d}(\mu, z) = \hdim \mu \leq \lmodbdim \mu \leq \umodbdim \mu = \dim_P \mu = \underset{z \sim \mu}{\mathrm{ess\ sup}}\ \overline{d}(\mu, z) \]
	($\dim_P$ denotes the packing dimension, see e.g.~\cite[Proposition 3.9]{falconer2014fractal} and \cite[Proposition 10.3]{FalconerTechniques}), the equality $\hdim \mu = \lmodbdim \mu$ holds for all \emph{exact dimensional} measures $\mu$, i.e.~the measures $\mu$ with $\underline{d}(\mu, z) = \overline{d}(\mu, z) = \mathrm{const}$ for $\mu$-almost every $z$.
\end{rem}

\bibliographystyle{alpha}
\bibliography{universal_bib}

\newcommand{\etalchar}[1]{$^{#1}$}
\def\cprime{$'$} \def\cprime{$'$}
\begin{thebibliography}{BAEFN93}

\bibitem[ABD{\etalchar{+}}19]{Riegler18}
G.~{Alberti}, H.~{B\"{o}lcskei}, C.~{De Lellis}, G.~{Koliander}, and
  E.~{Riegler}.
\newblock Lossless analog compression.
\newblock {\em IEEE Transactions on Information Theory}, 65(11):7480--7513,
  2019.

\bibitem[BAEFN93]{BAEFN93}
Asher Ben-Artzi, Alp Eden, Ciprian Foias, and Basil Nicolaenko.
\newblock H\"{o}lder continuity for the inverse of {M}a\~{n}\'{e}'s projection.
\newblock {\em J. Math. Anal. Appl.}, 178(1):22--29, 1993.

\bibitem[Ban51]{Banach51}
Stefan Banach.
\newblock {\em Wst\k{e}p do teorii funkcji rzeczywistych (Polish) [Introduction
  to the theory of real functions]}.
\newblock Monografie Matematyczne. Tom XVII. Polskie Towarzystwo Matematyczne,
  Warszawa-Wroc\l{}aw, 1951.

\bibitem[Bil99]{Billingsley99}
Patrick Billingsley.
\newblock {\em Convergence of probability measures}.
\newblock Wiley Series in Probability and Statistics: Probability and
  Statistics. John Wiley \& Sons, Inc., New York, second edition, 1999.
\newblock A Wiley-Interscience Publication.

\bibitem[Cab00]{CaballeroEmbed}
Victoria Caballero.
\newblock On an embedding theorem.
\newblock {\em Acta Math. Hungar.}, 88(4):269--278, 2000.

\bibitem[EFNT94]{EFNT94}
Alp Eden, Ciprian Foias, Basil Nicolaenko, and Roger~M. Temam.
\newblock {\em Exponential attractors for dissipative evolution equations},
  volume~37 of {\em RAM: Research in Applied Mathematics}.
\newblock Masson, Paris; John Wiley \& Sons, Ltd., Chichester, 1994.

\bibitem[ER85]{ER85}
Jean-Pierre Eckmann and David Ruelle.
\newblock Ergodic theory of chaos and strange attractors.
\newblock {\em Rev. Modern Phys.}, 57(3, part 1):617--656, 1985.

\bibitem[Fal97]{FalconerTechniques}
Kenneth Falconer.
\newblock {\em Techniques in fractal geometry}.
\newblock John Wiley \& Sons, Ltd., Chichester, 1997.

\bibitem[Fal14]{falconer2014fractal}
Kenneth Falconer.
\newblock {\em Fractal geometry}.
\newblock John Wiley \& Sons, Ltd., Chichester, third edition, 2014.
\newblock Mathematical foundations and applications.

\bibitem[FR02]{FR02}
Peter~K. Friz and James~C. Robinson.
\newblock Constructing an elementary measure on a space of projections.
\newblock {\em J. Math. Anal. Appl.}, 267(2):714--725, 2002.

\bibitem[GQS18]{GQS18}
Yonatan Gutman, Yixiao Qiao, and G\'{a}bor Szab\'{o}.
\newblock The embedding problem in topological dynamics and {T}akens' theorem.
\newblock {\em Nonlinearity}, 31(2):597--620, 2018.

\bibitem[GS00]{poly-interpolation}
Mariano Gasca and Thomas Sauer.
\newblock Polynomial interpolation in several variables.
\newblock {\em Adv. Comput. Math.}, 12(4):377--410, 2000.
\newblock Multivariate polynomial interpolation.

\bibitem[Gut16]{Gut16}
Yonatan Gutman.
\newblock Taken's embedding theorem with a continuous observable.
\newblock In {\em Ergodic theory}, pages 134--141. De Gruyter, Berlin, 2016.

\bibitem[GVL13]{MatrixComputations}
Gene~H. Golub and Charles~F. Van~Loan.
\newblock {\em Matrix computations}.
\newblock Johns Hopkins Studies in the Mathematical Sciences. Johns Hopkins
  University Press, Baltimore, MD, fourth edition, 2013.

\bibitem[HBS15]{HBS15}
Franz Hamilton, Tyrus Berry, and Timothy Sauer.
\newblock Predicting chaotic time series with a partial model.
\newblock {\em Phys. Rev. E}, 92:010902, Jul 2015.

\bibitem[HGLS05]{hgls05distinguishing}
Chih-Hao Hsieh, Sarah~M. Glaser, Andrew~J. Lucas, and George Sugihara.
\newblock Distinguishing random environmental fluctuations from ecological
  catastrophes for the {N}orth {P}acific {O}cean.
\newblock {\em Nature}, 435(7040):336--340, 2005.

\bibitem[HK99]{HK99}
Brian~R. Hunt and Vadim~Yu. Kaloshin.
\newblock Regularity of embeddings of infinite-dimensional fractal sets into
  finite-dimensional spaces.
\newblock {\em Nonlinearity}, 12(5):1263--1275, 1999.

\bibitem[Hoc14]{HochmanNotes}
Michael Hochman.
\newblock Lectures on dynamics, fractal geometry, and metric number theory.
\newblock {\em J. Mod. Dyn.}, 8(3-4):437--497, 2014.

\bibitem[HW41]{HW41}
Witold Hurewicz and Henry Wallman.
\newblock {\em Dimension {T}heory}.
\newblock Princeton Mathematical Series, v. 4. Princeton University Press,
  Princeton, N. J., 1941.

\bibitem[Kec95]{K95}
Alexander~S. Kechris.
\newblock {\em Classical descriptive set theory}, volume 156 of {\em Graduate
  Texts in Mathematics}.
\newblock Springer-Verlag, New York, 1995.

\bibitem[KY90]{KY90}
Eric~J. Kostelich and James~A. Yorke.
\newblock Noise reduction: finding the simplest dynamical system consistent
  with the data.
\newblock {\em Phys. D}, 41(2):183--196, 1990.

\bibitem[Mar54]{Marstrand}
John~M. Marstrand.
\newblock Some fundamental geometrical properties of plane sets of fractional
  dimensions.
\newblock {\em Proc. London Math. Soc. (3)}, 4:257--302, 1954.

\bibitem[Mat75]{Mattila-proj}
Pertti Mattila.
\newblock Hausdorff dimension, orthogonal projections and intersections with
  planes.
\newblock {\em Ann. Acad. Sci. Fenn. Ser. A I Math.}, 1(2):227--244, 1975.

\bibitem[Mat95]{mattila}
Pertti Mattila.
\newblock {\em Geometry of sets and measures in Euclidean spaces}, volume~44 of
  {\em Cambridge Studies in Advanced Mathematics}.
\newblock Cambridge University Press, Cambridge, 1995.
\newblock Fractals and rectifiability.

\bibitem[MGNS18]{SugiharaFishes}
Stephan~B. Munch, Alfredo Giron-Nava, and George Sugihara.
\newblock Nonlinear dynamics and noise in fisheries recruitment: A global
  meta-analysis.
\newblock {\em Fish and Fisheries}, 19(6):964--973, 2018.

\bibitem[Min70]{minty1970extension}
George~J. Minty.
\newblock On the extension of {L}ipschitz, {L}ipschitz-{H}{\"o}lder continuous,
  and monotone functions.
\newblock {\em Bulletin of the American Mathematical Society}, 76(2):334--339,
  1970.

\bibitem[Mn81]{Mane81}
Ricardo Ma\~{n}\'{e}.
\newblock On the dimension of the compact invariant sets of certain nonlinear
  maps.
\newblock In {\em Dynamical systems and turbulence, {W}arwick 1980 ({C}oventry,
  1979/1980)}, volume 898 of {\em Lecture Notes in Math.}, pages 230--242.
  Springer, Berlin-New York, 1981.

\bibitem[NV18]{NV18}
Raymundo Navarrete and Divakar Viswanath.
\newblock Prevalence of delay embeddings with a fixed observation function.
\newblock {\em Preprint. https://arxiv.org/abs/1806.07529}, 2018.

\bibitem[PCFS80]{PCFS80}
Norman~H. Packard, James~P. Crutchfield, J.~Doyne Farmer, and Robert~S. Shaw.
\newblock Geometry from a time series.
\newblock {\em Phys. Rev. Lett.}, 45:712--716, Sep 1980.

\bibitem[PdM82]{palis1982geometric}
Jacob Palis, Jr. and Welington de~Melo.
\newblock {\em Geometric theory of dynamical systems}.
\newblock Springer-Verlag, New York-Berlin, 1982.
\newblock An introduction, Translated from the Portuguese by A. K. Manning.

\bibitem[Rob05]{Rob05}
James~C. Robinson.
\newblock A topological delay embedding theorem for infinite-dimensional
  dynamical systems.
\newblock {\em Nonlinearity}, 18(5):2135--2143, 2005.

\bibitem[Rob11]{Rob11}
James~C. Robinson.
\newblock {\em Dimensions, embeddings, and attractors}, volume 186 of {\em
  Cambridge Tracts in Mathematics}.
\newblock Cambridge University Press, Cambridge, 2011.

\bibitem[Rud87]{R87}
Walter Rudin.
\newblock {\em Real and complex analysis}.
\newblock McGraw-Hill Book Co., New York, third edition, 1987.

\bibitem[SBDH97]{StarkEmbedSurvey}
Jaroslav Stark, David~S. Broomhead, Michael~Evan Davies, and Jeremy~P. Huke.
\newblock Takens embedding theorems for forced and stochastic systems.
\newblock In {\em Proceedings of the {S}econd {W}orld {C}ongress of {N}onlinear
  {A}nalysts, {P}art 8 ({A}thens, 1996)}, volume~30, pages 5303--5314, 1997.

\bibitem[SBDH03]{StarkStochEmbed}
Jaroslav Stark, David~S. Broomhead, Michael~Evan Davies, and Jeremy~P. Huke.
\newblock Delay embeddings for forced systems. {II}. {S}tochastic forcing.
\newblock {\em J. Nonlinear Sci.}, 13(6):519--577, 2003.

\bibitem[SGM90]{sgm90distinguishing}
George Sugihara, Bryan Grenfell, and Robert May.
\newblock Distinguishing error from chaos in ecological time-series.
\newblock {\em Philosophical Transactions of the Royal Society B-Biological
  Sciences}, 330(1257):235--251, 1990.

\bibitem[SM90]{sm90nonlinear}
George Sugihara and Robert May.
\newblock Nonlinear forecasting as a way of distinguishing chaos from
  measurement error in time series.
\newblock {\em Nature}, 344(6268):734--741, 1990.

\bibitem[Sta99]{1999delay}
Jaroslav Stark.
\newblock Delay embeddings for forced systems. {I}. {D}eterministic forcing.
\newblock {\em J. Nonlinear Sci.}, 9(3):255--332, 1999.

\bibitem[SY97]{SauerYorke97}
Timothy~D. Sauer and James~A. Yorke.
\newblock Are the dimensions of a set and its image equal under typical smooth
  functions?
\newblock {\em Ergodic Theory Dynam. Systems}, 17(4):941--956, 1997.

\bibitem[SYC91]{SYC91}
Timothy~D. Sauer, James~A. Yorke, and Martin Casdagli.
\newblock Embedology.
\newblock {\em J. Statist. Phys.}, 65(3-4):579--616, 1991.

\bibitem[Tak81]{T81}
Floris Takens.
\newblock Detecting strange attractors in turbulence.
\newblock In {\em Dynamical systems and turbulence, {W}arwick 1980}, volume 898
  of {\em Lecture Notes in Math.}, pages 366--381. Springer, Berlin-New York,
  1981.

\bibitem[Vos03]{Voss03}
Henning~U. Voss.
\newblock Synchronization of reconstructed dynamical systems.
\newblock {\em Chaos}, 13(1):327--334, 2003.

\bibitem[Whi36]{Whitney36}
Hassler Whitney.
\newblock Differentiable manifolds.
\newblock {\em Ann. of Math. (2)}, 37(3):645--680, 1936.

\bibitem[Yor69]{y69}
James~A. Yorke.
\newblock Periods of periodic solutions and the {L}ipschitz constant.
\newblock {\em Proc. Amer. Math. Soc.}, 22:509--512, 1969.

\end{thebibliography}
 
\end{document}